\documentclass[11pt]{article}
\usepackage{amsmath,amssymb,amsthm}
\usepackage{mathtools}
\usepackage{graphics,epsfig}
\usepackage{hyperref}
\usepackage{natbib}
\usepackage{color}
\usepackage{graphicx}
\usepackage{caption}
\usepackage{subcaption}
\usepackage{float}
\usepackage{mathrsfs}
\usepackage{multirow}
\usepackage{comment}
\usepackage{setspace}
\usepackage{bbm}
\usepackage{bm}
\usepackage{amsfonts}
\usepackage{xcolor}
\usepackage{pslatex}
\usepackage{setspace}
\usepackage{bbm}
\usepackage{booktabs}

\newtheorem{remark}{\bf Remark}
\newtheorem{rmk}[remark]{\bf Remark}

\newtheorem{theorem}{\bf Theorem}

\newtheorem{prop}[theorem]{\bf Proposition}


\begin{document}

	\title{\bfseries Optimal Harvesting under Stochastic Control: HJB Equation and Feynman–Kac Representation }

	\author{
		Paramahansa Pramanik$^{1,3}$ \and 
		Fatamatuj Johora$^{2}$
	}
	
	\date{
		\small
		$^{1}$ Department of Mathematics and Statistics, University of South Alabama, Mobile, AL 36688, United States.\\
		\texttt{fj2521@jagmail.southalabama.edu}\\
		\vspace{0.5em}
		$^{2}$ Department of Mathematics and Statistics, University of South Alabama, Mobile, AL 36688, United States.\\
		$^{3}$Corresponding author, \texttt{ppramanik@southalabama.edu}
	}
	\maketitle
	
	\begin{abstract}
		Sustainable resource management requires harvesting strategies that account for environmental variability and ecological uncertainty. This study investigates optimal harvesting of renewable biological resources within a stochastic framework, where population dynamics are influenced by random environmental fluctuations and modeled using stochastic differential equations. Two complementary approaches are employed: the Hamilton–Jacobi–Bellman (HJB) equation and the Feynman–Kac representation. The HJB framework provides a dynamic optimization rule and characterizes the value function through a nonlinear partial differential equation, while the Feynman–Kac approach offers a probabilistic interpretation of expected returns. A comparative analysis demonstrates the theoretical consistency and practical relevance of both methods for designing economically efficient and ecologically sustainable harvesting policies under uncertainty.
		
	\end{abstract}
	
	
	\section{Introduction:}
	In this study, we investigate the problem of determining an optimal harvesting strategy for a renewable biological resource when population growth is influenced by unpredictable environmental variation. Natural systems rarely evolve in a perfectly stable manner; changes in climate conditions, habitat quality, resource availability, and other ecological factors can cause population levels to fluctuate over time. To represent this uncertainty, we model the resource dynamics using a generalized growth combined with stochastic variation, allowing the population to experience random deviations from its average growth trajectory \citep{pramanik2020optimization,pramanik2023semicooperation}. The objective is to identify a harvesting policy that balances economic return with long-term resource persistence by maximizing the expected discounted profit accumulated over a finite planning horizon. Such a framework reflects the practical challenges faced in fisheries, forestry, wildlife management, and other environmental systems where decision makers must act despite incomplete knowledge of future conditions \citep{kakkat2026angiotensin}. To address this problem, we employ two complementary mathematical approaches from stochastic control theory. The first is the Hamilton-Jacobi-Bellman (HJB) equation, which provides a systematic method for identifying the harvesting effort that yields the greatest expected benefit at every point in time. Through this approach, the optimization problem is translated into a value function that measures the best attainable future return from any given population state \citep{pramanik2020motivation}. The HJB formulation offers a clear and structured description of how harvesting decisions should adapt to changing ecological conditions and uncertainty. By connecting population dynamics, economic rewards, and future expectations within a single framework, it provides valuable insight into the trade-offs between immediate extraction and resource conservation.
	
	Alongside the HJB formulation, we examine a probabilistic perspective based on the Feynman-Kac \citep{kac_1949} representation and its associated path-integral interpretation. Rather than focusing solely on the evolution of an optimal value function through a governing equation, this approach considers the collection of all possible future population trajectories that may arise under environmental uncertainty. Each potential path contributes to the overall expected return, with more favorable outcomes exerting a greater influence on the final valuation. In this sense, the path-integral approach provides an intuitive description of decision making under uncertainty by viewing the value of a harvesting policy as the cumulative effect of many possible ecological futures rather than a single predicted outcome. This interpretation is particularly appealing in environmental applications because natural systems often exhibit substantial variability, making it important to account for a broad range of possible scenarios \citep{pramanik2021optimala,pramanik2021scoring}. The Feynman-Kac representation therefore complements the deterministic optimization perspective of the HJB equation by revealing how uncertainty propagates through the system and influences long-term economic performance. A central goal of this study is to compare these two viewpoints and demonstrate that, despite their distinct conceptual foundations, they lead to the same optimal harvesting strategy. By highlighting both their similarities and differences, we show that the HJB and Feynman-type path-integral approaches provide consistent and mutually reinforcing tools for analyzing sustainable resource management in uncertain environments. Together, these methods offer a richer understanding of how ecological variability, economic objectives, and management decisions interact in the design of effective harvesting policies \citep{pramanik2024estimation,vikramdeo2024mitochondrial}.
	
	The motivation for this study arises from the increasing challenges associated with managing renewable natural resources in a world characterized by environmental variability and economic uncertainty. Many traditional harvesting models are built on simplifying assumptions that treat population growth as predictable and market conditions as relatively stable. While such assumptions are useful for developing basic theoretical insights, they often fail to capture the realities faced by resource managers \citep{pramanik2024motivation,bulls2025assessing}. Fisheries, forests, wildlife populations, and other renewable resources are continually influenced by changing environmental conditions, including fluctuations in temperature, rainfall, habitat quality, disease outbreaks, and other ecological disturbances. At the same time, economic conditions can shift unexpectedly due to variations in consumer demand, production costs, policy changes, and broader market forces. These interacting sources of uncertainty make resource management far more complex than suggested by deterministic models. Decisions that appear profitable under average conditions may become unsustainable when environmental shocks reduce population abundance, while overly conservative strategies may sacrifice substantial economic benefits when favorable conditions occur. Consequently, there is a growing need for analytical frameworks that explicitly recognize uncertainty as an inherent feature of natural resource systems rather than treating it as a secondary consideration. By incorporating random fluctuations directly into the modeling process, stochastic approaches provide a more realistic representation of how ecological and economic systems evolve through time \citep{hertweck2023clinicopathological,khan2023myb}. Such models can help identify harvesting policies that remain effective across a wide range of possible future conditions, thereby supporting management strategies that are both economically beneficial and environmentally responsible. This perspective is particularly important in the context of global environmental change, where increasing climatic variability and human pressures are making future resource dynamics more difficult to predict.
	
	To address these challenges, this study develops a comprehensive framework for analyzing optimal harvesting under uncertainty through the use of stochastic control methods. The central objective is not only to determine how harvesting effort should be adjusted in response to changing resource conditions, but also to understand how different mathematical perspectives can be used to evaluate management decisions in uncertain environments. Particular attention is given to the comparison between the HJB equation and the Feynman-Kac representation. Although these methods originate from different conceptual viewpoints, both are designed to evaluate the long-term consequences of decisions made under uncertainty \citep{kakkat2023cardiovascular,khan2023myb}. The HJB framework provides a structured optimization perspective by identifying the harvesting policy that maximizes expected future benefits while accounting for the changing condition of the resource. In contrast, the Feynman-Kac approach views the problem through the lens of possible future system trajectories, linking management outcomes to the collection of ecological paths that may emerge under random environmental influences. Examining these approaches together offers valuable insight into the relationship between deterministic optimization and probabilistic interpretation. Rather than treating uncertainty as a complication to be avoided, both methods incorporate it directly into the decision-making process, allowing managers to evaluate risks and opportunities in a systematic manner. Through this comparative framework, the study seeks to demonstrate how modern stochastic methods can contribute to the development of harvesting strategies that maintain ecological resilience while supporting sustainable economic returns \citep{kakkat2023cardiovascular,khan2024mp60}. By bridging applied mathematics and environmental science, the analysis provides a practical foundation for understanding how uncertainty shapes resource management decisions and how robust harvesting policies can be designed in the presence of unpredictable ecological and economic conditions \citep{maki2025new}.
	
	By combining the HJB equation with the Feynman-Kac representation, this study seeks to develop a unified understanding of optimal harvesting in environments where both ecological conditions and economic circumstances change unpredictably over time. Resource management decisions are rarely made in settings where all relevant information is known with certainty. Instead, managers must often respond to shifting population levels, changing environmental conditions, and fluctuations in the economic value of harvested products. These uncertainties can interact in complex ways, making it difficult to identify harvesting strategies that remain effective over long periods \citep{vikramdeo2024abstract,vikramdeo2023profiling}. A mathematical framework that incorporates both ecological variability and economic uncertainty therefore offers significant advantages for understanding how sustainable management decisions should be made. The HJB equation provides a structured method for identifying optimal actions by evaluating the long-term consequences of present decisions, while the Feynman-Kac representation offers a complementary viewpoint by considering the many possible future paths that a resource system may follow. Together, these approaches allow the optimization problem to be examined from different but consistent perspectives. The resulting framework provides both analytical clarity and practical insight, helping to explain not only what decisions should be made but also why those decisions remain effective under uncertain conditions \citep{pramanik2024bayes}. Such an integrated perspective is valuable because it reflects the reality that environmental systems are influenced by a wide range of factors that cannot be predicted perfectly, and therefore management strategies must be designed with flexibility and resilience in mind.
	
	The importance of this issue extends far beyond theoretical interest because the sustainable use of renewable resources has direct ecological, economic, and social consequences. Decisions regarding harvesting influence not only current economic returns but also the future condition of ecosystems that support livelihoods, biodiversity, and environmental services. In fisheries, for example, excessive harvesting can reduce stock abundance and compromise long-term productivity, while overly restrictive harvesting policies may impose unnecessary economic burdens on communities that depend on fishing activities \citep{khan2024mp60,dasgupta2023frequent,hertweck2023clinicopathological}. Similar challenges arise in forestry, wildlife management, and other renewable resource sectors, where managers must continually balance present benefits against future sustainability. Effective decision-making therefore requires an understanding of both biological growth processes and the uncertain conditions under which those processes occur. Environmental disturbances, changing climatic patterns, habitat degradation, and fluctuations in resource demand can all alter the outcomes of management decisions in ways that are difficult to anticipate using deterministic models alone. By explicitly incorporating uncertainty into the analysis, the framework developed in this study provides a more realistic representation of the conditions under which resource managers operate. This approach recognizes that sustainable management is not simply a matter of maximizing short-term gains, but rather involves identifying policies that can perform well across a broad range of possible future scenarios \citep{pramanik2024estimation,pramanik2023cont}. As a result, the study contributes to a deeper understanding of how uncertainty should be incorporated into the design of harvesting strategies intended to protect both ecological integrity and economic viability.
	
	The broader significance of this work lies in its potential to serve as a foundation for future research and practical applications in environmental management. Although the present analysis focuses on a single-resource harvesting problem, the ideas developed here can be extended to more complex ecological systems involving multiple interacting species, competing resource users, and additional environmental pressures \citep{pramanik2024estimation1,yusuf2025prognostic}. Real ecosystems are often characterized by intricate relationships among species, where changes in one population may influence the growth, survival, or productivity of others. Incorporating such interactions into stochastic harvesting models represents an important direction for future investigation \citep{pramanik2023path}. Likewise, the framework can be adapted to address larger management challenges in fisheries regulation, forest planning, wildlife conservation, and ecosystem restoration, where uncertainty plays a central role in determining long-term outcomes. The integration of optimization methods with probabilistic representations offers a flexible platform for exploring these more complicated settings while maintaining a clear connection between ecological dynamics and management objectives. Ultimately, the study contributes to the development of more robust and adaptable decision-support tools capable of guiding policy and management under uncertain conditions. As environmental and economic systems become increasingly variable, the need for approaches that can accommodate unpredictability while promoting sustainability will continue to grow \citep{pramanik2023cmbp,yusuf2025predictive}. The methods examined here represent an important step toward achieving that goal by providing a framework through which informed and resilient resource management decisions can be developed and evaluated \citep{vikramdeo2024abstract}. 
	
	\begin{figure}[htbp]
		\centering
		\includegraphics[width=\textwidth]{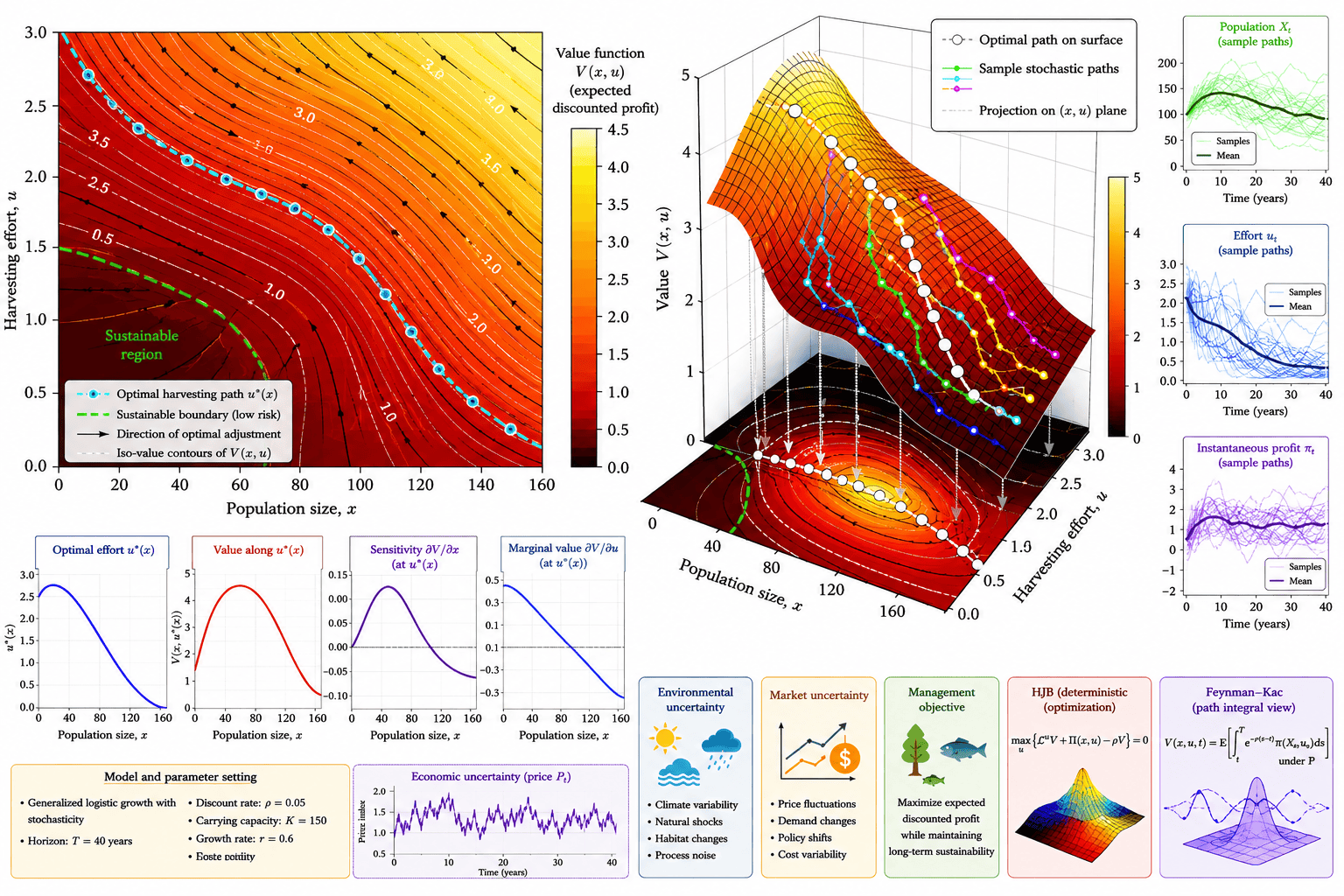}
		\caption{Optimal harvesting under stochastic ecological and economic uncertainty using heatmap and three-dimensional value function representations.}
		\label{fig:optimal_harvesting_infographic}
	\end{figure}
	
	Figure \ref{fig:optimal_harvesting_infographic} provides a comprehensive visualization of the stochastic optimal harvesting framework by combining deterministic optimization and probabilistic analysis within a single graphical representation. The left panel presents a heat map of the value function, where warmer colors indicate higher expected discounted profits under different combinations of population size and harvesting effort. Contour lines, directional arrows, and the highlighted optimal harvesting path illustrate how management decisions should adjust as resource abundance changes, while the sustainable region identifies states where long-term resource persistence is more likely. The right panel displays the corresponding three-dimensional value surface, revealing how expected economic returns vary across the state space and highlighting the optimal trajectory on the surface. Sample stochastic paths demonstrate the influence of environmental variability on population dynamics and harvesting decisions, emphasizing that future outcomes are uncertain rather than fixed. The smaller panels further summarize the behavior of optimal effort, value sensitivity, population trajectories, profit evolution, and economic uncertainty. Together, the figure illustrates how the HJB equation identifies optimal harvesting policies, while the Feynman-Kac representation interprets the value function as an average over many possible future ecological and economic pathways, providing a unified framework for sustainable resource management under uncertainty.

	\subsection{Research Background:}
	
	The problem considered in this study has attracted considerable attention across the fields of ecology, stochastic control, and environmental economics because of its direct relevance to the sustainable management of renewable natural resources. The central challenge is to identify harvesting strategies that generate stable economic benefits while preserving the long-term viability of the underlying resource population. Achieving this balance is particularly important in systems such as fisheries, forests, and wildlife populations, where excessive extraction can compromise future productivity and ecological stability \citep{pramanik2021consensus}. Early mathematical models of harvesting were largely deterministic in nature, assuming that population growth followed a predictable trajectory within an unchanging environment and that economic conditions remained relatively stable. Although these assumptions simplified analysis and provided valuable theoretical insights, they did not adequately represent the uncertainty that characterizes real-world ecological and economic systems. Natural populations are continuously affected by environmental variability, including changes in climate, habitat quality, food availability, disease outbreaks, and other factors that influence growth and survival. Likewise, market conditions can fluctuate because of shifts in demand, production costs, policy interventions, and broader economic trends \citep{pramanik2023cmbp,yusuf2025predictive}. As a result, resource managers must often make decisions without complete knowledge of future conditions. Over time, researchers recognized the limitations of deterministic frameworks and increasingly adopted stochastic approaches that explicitly account for uncertainty. This transition led to the development of models capable of representing the random nature of ecological processes and their implications for resource management decisions. Consequently, stochastic harvesting models have become an important area of research, providing more realistic tools for evaluating management strategies in uncertain environments. The present review examines the major developments within this literature, focusing on the methods that have been used to study harvesting under uncertainty, the key findings that have emerged from previous investigations, and the challenges that continue to motivate further research in stochastic control. A major advance in this field was the introduction of stochastic differential equations (SDEs) into models of population dynamics and resource management. Researchers incorporated random processes into classical growth models to capture the influence of environmental variability on biological populations. In particular, the logistic growth model, one of the most widely used representations of renewable resource dynamics, was extended to include stochastic fluctuations arising from variations in birth rates and environmental conditions \cite{brites_braumann_2023}. Within these frameworks, uncertainty is commonly represented through Brownian motion or Wiener processes, allowing the population trajectory to evolve in a random yet mathematically tractable manner. The resulting harvesting problem is naturally formulated as a stochastic control problem in which management decisions influence both current outcomes and future resource states. Among the foundational contributions in this area, the work of \cite{lungu_oksendal_2001} demonstrated how stochastic control theory could be applied to optimal harvesting problems and established the HJB equation as a powerful tool for determining optimal management policies under uncertainty. Their results showed that stochastic control methods provide a rigorous framework for balancing economic objectives with ecological sustainability when resource dynamics are influenced by random environmental fluctuations.
	
	Recent developments in stochastic harvesting research have expanded beyond purely ecological uncertainty to incorporate the effects of economic variability, recognizing that management decisions are influenced not only by changes in resource abundance but also by fluctuations in market conditions. In particular, the studies of \cite{brites_braumann_2023} and \cite{hening_tran_2020} examined optimal harvesting problems in settings where resource prices evolve randomly through time rather than remaining fixed or perfectly predictable. Their analyses demonstrated that price uncertainty can play a crucial role in determining harvesting behavior because changes in market value directly affect the trade-off between immediate extraction and future resource preservation. When prices are volatile, harvesting decisions that appear optimal under constant-price assumptions may no longer maximize long-term economic returns. As a result, incorporating stochastic price dynamics into harvesting models provides a more realistic representation of the conditions faced by resource managers and improves our understanding of how economic and ecological factors jointly influence management outcomes. These studies highlighted that optimal harvesting policies depend not only on biological growth and environmental variability but also on the timing and magnitude of market fluctuations. Consequently, they contributed significantly to the growing literature that seeks to integrate ecological processes and economic uncertainty within a unified stochastic framework. Despite these important advances, several limitations remain. A major restriction of much of the existing literature is its emphasis on single-species systems, where population dynamics are represented by a single biological stock interacting with a relatively simple environment. Although such models provide valuable theoretical insights and are often more mathematically tractable, they do not fully capture the complexity of natural ecosystems. In reality, renewable resources rarely exist in isolation. Species frequently compete for food, habitat, and other essential resources, while predator--prey relationships, mutualistic interactions, and broader ecosystem processes can substantially influence population dynamics. These ecological interdependencies introduce additional layers of complexity into the formulation of optimal harvesting problems because management decisions affecting one species may have indirect consequences for many others. The work of \cite{lungu_oksendal_2001} extended stochastic harvesting theory to systems involving multiple populations and represented an important step toward addressing this challenge. However, these models generally relied on relatively simplified assumptions regarding species interactions and often represented ecological relationships using approximately linear forms \citep{feynman_1949,bellman_1957}. While such assumptions facilitate analysis, they may not adequately describe the highly nonlinear responses commonly observed in natural ecosystems. Population growth, competition, predation, and environmental feedback mechanisms often interact in complex ways that can generate outcomes not predicted by simplified models. Consequently, many important ecological processes remain insufficiently explored within existing stochastic harvesting frameworks. This gap in the literature underscores the need for more comprehensive models capable of representing nonlinear interactions among multiple species while simultaneously accounting for environmental variability and economic uncertainty \citep{pramanik2021consensus}. Developing such models would improve the realism of stochastic resource management frameworks and provide deeper insight into how harvesting policies should be designed in complex ecological systems where biological and economic processes are tightly interconnected.
	
	The HJB equation has become one of the most widely used mathematical tools for solving stochastic optimal control problems and has played a central role in the analysis of optimal harvesting under uncertainty. Within this framework, the management problem is expressed through a value function that represents the maximum expected benefit obtainable from a given resource state when future decisions are chosen optimally \citep{pramanik2021,pramanik2022stochastic}. The HJB equation provides a dynamic description of this value function by linking current management decisions with their future consequences, thereby offering a systematic way to determine how harvesting effort should evolve through time in response to changing ecological conditions and random environmental influences. In the context of renewable resource management, the HJB equation serves as a mathematical representation of the balance between immediate economic gains from harvesting and the long-term benefits associated with maintaining a healthy resource population. This feature has made the approach particularly attractive for studying single-species systems subject to stochastic growth and harvesting pressures, and it has been applied extensively in the literature to examine how uncertainty affects resource exploitation strategies \citep{hening_tran_2020}. One of the principal strengths of the HJB framework is its ability to summarize both ecological and economic considerations within a single optimization structure. The value function captures the trade-offs between current extraction, future population growth, and expected economic returns, allowing managers to evaluate the long-term implications of harvesting decisions. By solving the HJB equation, it is possible to derive the harvesting effort that maximizes expected benefits at each point in time while accounting for uncertainty in the underlying biological system. This capability provides important insight into how optimal policies should respond to fluctuations in resource abundance and environmental conditions. Despite these advantages, the HJB framework also faces several practical and computational challenges. A major difficulty arises from the complexity of the equation itself, particularly when the model contains nonlinear relationships or multiple sources of uncertainty \citep{pramanik2023optimization001}. Realistic resource management problems often involve interactions among biological dynamics, environmental variability, and economic factors such as changing market conditions, all of which can substantially increase the complexity of the optimization problem. In such settings, obtaining analytical solutions is rarely possible, making numerical approximation methods necessary. Common approaches include finite-difference schemes and the Crank--Nicolson method, which approximate the value function over a discretized state space. Although these techniques can produce useful results, they often require substantial computational resources, especially when the number of state variables increases. This issue becomes particularly severe in high-dimensional systems involving multiple species, multiple interacting populations, or several stochastic processes operating simultaneously. As the dimension of the state space grows, the computational burden increases dramatically, making numerical implementation increasingly difficult and, in some cases, impractical \citep{imai2010identification,pramanik2025construction,pramanik2025optimal}. Consequently, obtaining accurate solutions for large-scale stochastic harvesting problems remains a significant challenge. These limitations have motivated continued research into alternative formulations, approximation techniques, and computational strategies that can preserve the theoretical strengths of the HJB framework while improving its applicability to complex ecological and environmental systems.
	
	The Feynman-Kac formula, in comparison to the HJB equation, gives us a probabilistic answer to optimal control questions. The Feynman-Kac formula demonstrates the value function as an assumption over the stochastic paths of the structure. The Feynman-Kac method gives a way to compute total anticipated profits by attaching the arbitrary shifts in a population to the particular harvesting approach used. This type of probabilistic explanation has the benefit of giving a more instinctive perception of the optimal harvesting strategy \citep{pramanik2024stochastic,pramanik2025stubbornness}. The Feynman--Kac representation has become an important tool for examining how environmental uncertainty influences optimal harvesting decisions and for providing a probabilistic interpretation of resource management problems. Rather than focusing directly on a governing partial differential equation, this approach links the value function to the expected outcome of a stochastic process, thereby connecting future economic returns to the many possible paths that a biological system may follow through time. In the context of renewable resource management, this perspective offers a natural way to evaluate the effects of environmental variability on harvesting strategies because it explicitly incorporates uncertainty into the calculation of future benefits. The work of \cite{hening_tran_2020} demonstrated the usefulness of the Feynman-Kac representation in modeling harvesting and seeding decisions under stochastic conditions, showing that the value function can be interpreted as the expected stream of future profits generated by a particular management strategy. This interpretation provides an intuitive understanding of decision-making under uncertainty by viewing management outcomes as averages over many possible ecological futures rather than as the result of a single predicted trajectory. One of the principal advantages of the Feynman--Kac representation is that it transforms a stochastic control problem into an expectation involving a stochastic process, which can often simplify both interpretation and computation. By connecting optimization problems with probabilistic descriptions of system behavior, the method offers valuable insight into the relationship between ecological variability and economic performance. Nevertheless, the approach is not without limitations \citep{pramanik2025dissecting,pramanik2025factors}. A common restriction is that many applications of the Feynman-Kac representation rely on relatively simple relationships between control variables and the underlying stochastic processes. While such assumptions can facilitate analysis, they may not adequately represent more complex environmental and economic systems in which interactions are strongly nonlinear. Real-world resource management problems often involve feedback mechanisms, thresholds, and other nonlinear effects that cannot be captured easily within simplified formulations. In addition, many implementations of the Feynman--Kac representation assume a degree of independence among stochastic processes that may not be realistic in practice \citep{pramanik2021thesis,pramanik2016}. Environmental conditions, biological responses, and economic factors are frequently interconnected, and changes in one component can influence the behavior of others through intricate feedback relationships. As a result, assumptions of weak dependence or independence may limit the applicability of some models when ecological and economic systems are tightly coupled. Beyond these methodological challenges, an important limitation shared by much of the existing literature is the strong emphasis on single-species harvesting models. Many stochastic harvesting studies assume the presence of a single resource population and therefore neglect the broader ecological interactions that characterize natural systems \citep{hua2019assessing,polansky2021motif}. In reality, ecosystems are commonly composed of multiple species linked through competition, predation, mutual dependence, and shared resource use. These interactions can substantially influence population dynamics and, consequently, the effectiveness of harvesting strategies. Extending existing stochastic harvesting frameworks to include multiple interacting species would represent a significant advancement in the field because it would allow researchers to examine management decisions within a more realistic ecological setting. Such developments could improve our understanding of how harvesting policies affect interconnected populations and provide deeper insight into sustainable resource management in ecosystems where ecological relationships play a critical role in determining long-term outcomes.
	
	An important direction for future research involves expanding stochastic harvesting models to incorporate a broader range of economic, computational, and social considerations that influence real-world resource management decisions. Although several studies have recognized the importance of price uncertainty and have incorporated stochastic price dynamics into harvesting frameworks, many of these models continue to focus primarily on market volatility while giving limited attention to the institutional and policy environments within which harvesting activities take place \citep{pramanik2024estimation,pramanik2023cont}. In practice, resource users operate under a variety of regulations, market interventions, and government programs that can substantially affect harvesting behavior and long-term resource outcomes. Examples include harvest quotas, licensing requirements, seasonal closures, trade restrictions, taxation policies, subsidy programs, and other administrative measures designed to balance economic development with environmental protection. These policy instruments often influence decision-making as strongly as ecological conditions or market prices, yet they remain underrepresented in many mathematical models. Incorporating such factors into stochastic harvesting frameworks would lead to a more realistic description of resource management systems by accounting not only for uncertainty in biological growth and market conditions but also for the policy mechanisms that shape human behavior \citep{pramanik2025strategic,pramanik2025impact}. This extension would be particularly valuable in sectors such as fisheries and forestry, where regulatory decisions frequently determine access to resources, influence harvesting intensity, and affect the economic viability of harvesting operations. Models that integrate ecological uncertainty, market fluctuations, and policy interventions within a unified framework could therefore provide more useful guidance for managers and policymakers seeking to design sustainable and effective resource management strategies. In addition to these economic and institutional considerations, significant opportunities remain for advancing the computational methods used to solve stochastic harvesting problems. Both the HJB equation and the Feynman--Kac representation can become computationally demanding when applied to realistic systems involving multiple state variables, nonlinear interactions, or several sources of uncertainty \citep{pramanik2025strategies,pramanik2023optimization001}. Numerical implementation becomes increasingly challenging as the dimension of the problem grows, and traditional methods may require substantial computational resources while still facing difficulties in maintaining accuracy \citep{carey2006race}. Consequently, there is considerable scope for developing more efficient numerical algorithms and approximation techniques capable of handling large-scale stochastic control problems without sacrificing reliability. Emerging computational approaches, including machine learning methods, data-driven optimization techniques, and approximate dynamic programming, offer promising opportunities for overcoming some of these challenges. Such methods may provide new ways to approximate value functions, identify near-optimal policies, and analyze complex systems that are difficult to address using conventional numerical procedures. Improvements in computational efficiency would make it possible to investigate richer and more realistic ecological models while preserving the analytical rigor needed for effective decision support \citep{pramanik2024measuring,pramanik2024dependence}. Beyond economic and computational issues, another area that remains insufficiently explored is the incorporation of social impacts and policy constraints into stochastic harvesting models. Most existing studies focus primarily on maximizing economic returns or maintaining biological sustainability, often overlooking the broader social consequences of management decisions. In reality, harvesting policies can have significant effects on employment, community stability, income distribution, and regional development. For example, restrictive harvesting measures designed to protect resource populations may reduce short-term employment opportunities or disrupt communities that depend heavily on natural resource industries. Conversely, excessive harvesting may generate immediate economic gains while creating long-term social and environmental costs. Likewise, practical resource management is often constrained by political realities, legal requirements, and institutional objectives that cannot be captured fully by purely economic optimization criteria. Policy instruments such as quota systems, harvest limits, protected areas, and conservation mandates frequently impose restrictions on harvesting decisions that must be considered when evaluating management strategies. Incorporating these social and policy dimensions into future stochastic harvesting models would support the development of management frameworks that are not only economically efficient and ecologically sustainable but also socially acceptable and politically feasible. Such an integrated perspective would provide a more comprehensive foundation for decision-making and would better reflect the complex set of objectives that characterize modern environmental management \citep{pramanik2024parametric,pramanik2025dissecting}.

	\subsection{Main Contributions:}
	This study contributes to the growing body of research on optimal harvesting under uncertainty by bringing together the HJB equation and the nonlinear Feynman--Kac representation within a unified stochastic control framework and examining their relationship in a common ecological setting. Although both approaches have been widely used in the analysis of stochastic optimization problems, they are often presented independently, with limited attention given to their direct comparison within the same harvesting model. The present work addresses this gap by demonstrating how these two methodologies provide complementary perspectives on the same underlying problem. The HJB equation offers a dynamic optimization framework in which the optimal harvesting policy is obtained through a partial differential equation that characterizes the value function and describes how management decisions should evolve in response to changes in the resource state. In contrast, the nonlinear Feynman--Kac representation provides a probabilistic viewpoint by expressing the same value function in terms of the expected behavior of stochastic system trajectories. By examining these approaches side by side, the study highlights the conceptual connection between deterministic optimization and probabilistic interpretation, showing that both methods ultimately describe the same optimal harvesting strategy despite originating from different mathematical viewpoints. A distinguishing feature of the analysis is its emphasis on the ecological dimensions of the harvesting problem. To maintain a clear focus on biological dynamics and environmental variability, the model intentionally excludes stochastic price processes and concentrates instead on the effects of population movement, harvesting effort, and random environmental fluctuations. This simplification allows the analysis to isolate the role of ecological uncertainty and provides a clearer understanding of how resource dynamics influence optimal management decisions. In addition to developing the stochastic control framework, the study establishes the mathematical validity of the controlled model by presenting existence and uniqueness results for the underlying stochastic differential equation. These results ensure that the state process is well defined and that the model possesses a unique solution under appropriate conditions, thereby providing a rigorous foundation for the subsequent optimization analysis. The combination of theoretical development and empirical implementation strengthens the contribution of the work by linking mathematical analysis with practical resource-management applications. Furthermore, the study serves as a foundation for future investigations aimed at incorporating additional sources of complexity that are common in real-world resource systems. Potential extensions include the introduction of stochastic price dynamics, interactions among multiple species, market regulations, harvesting restrictions, and broader policy constraints that influence management decisions. By establishing a coherent framework that connects the HJB equation and the nonlinear Feynman--Kac representation while focusing on ecological uncertainty, this work provides a useful starting point for the development of more comprehensive stochastic harvesting models capable of addressing the increasingly complex challenges associated with sustainable resource management.
	
	\section{Theoretical Background:}
	
	This section employs stochastic control theory as a framework for investigating optimal harvesting decisions in renewable resource systems that are influenced by environmental uncertainty and random fluctuations. The central objective is to determine how harvesting effort should be adjusted through time in order to balance economic returns with the long-term sustainability of the resource population when future conditions cannot be predicted with certainty \citep{pramanik2025optimal1,powell2025genomic}. To achieve this goal, the analysis combines three closely related components: a controlled stochastic differential equation, the HJB equation, and the nonlinear Feynman--Kac representation. The controlled stochastic differential equation serves as the foundation of the model by describing how the resource state evolves over time under the combined influence of natural growth, harvesting activity, and environmental variability. In this formulation, harvesting effort enters directly into the dynamics of the resource population, allowing management decisions to influence future resource abundance while environmental disturbances introduce random deviations from the average growth pattern. As a result, the state process reflects the reality that renewable resources such as fisheries, forests, and wildlife populations are shaped simultaneously by human intervention and unpredictable ecological conditions. Building upon this dynamic representation, the HJB equation is used to characterize the optimal harvesting policy \citep{pramanik2026strategic}. The HJB framework links present decisions with future consequences through a value function that measures the maximum expected benefit attainable from any given resource state. By solving this optimization problem, it becomes possible to identify harvesting strategies that account not only for immediate economic gains but also for their impact on future resource availability and expected returns. The HJB equation therefore provides a systematic description of how optimal decisions should adapt to changing ecological conditions and uncertainty over time. Complementing this optimization perspective, the nonlinear Feynman--Kac representation offers a probabilistic interpretation of the same value function \citep{pramanik2026quantum}. Rather than focusing solely on the governing equation satisfied by the value function, this approach views the problem through the collection of possible future trajectories that the resource system may follow. In this representation, the value function can be interpreted as an expected quantity obtained by averaging over many potential stochastic paths, each reflecting a different realization of environmental variability and management outcomes \citep{pramanik2024analysis,dunbar2026modeling}. This probabilistic viewpoint provides additional insight into the role of uncertainty by illustrating how future ecological conditions contribute to the overall value of a harvesting strategy. Although the HJB equation and the nonlinear Feynman--Kac representation originate from different mathematical perspectives, they ultimately describe the same optimization problem and yield consistent characterizations of the value function and optimal harvesting policy. Their integration within a single stochastic control framework therefore provides both a dynamic optimization perspective and a probabilistic interpretation of resource management under uncertainty \citep{dasgupta2026frequent}. By combining these complementary approaches, the study establishes a comprehensive methodology for analyzing renewable resource harvesting in stochastic environments, offering a clearer understanding of how ecological variability, management decisions, and long-term economic objectives interact within a unified mathematical structure.
	
	\subsection{Stochastic Harvesting Model:}
	Let $(\Omega,\mathcal{F},\{\mathcal{F}_t\}_{t\geq0},\mathbb{P})$ be a filtered probability space satisfying the usual conditions, where $\{\mathcal{F}_t\}_{t\geq0}$ represents the information available up to time $t$ and supports a standard one-dimensional Brownian motion $\{W(t)\}_{t\geq0}$. Within this stochastic framework, let $X(t)$ denote the state of a renewable biological resource at time $t$, interpreted either as the underlying stock level or, in the empirical application, as a normalized catch index. The evolution of the resource is assumed to be governed by both deterministic biological mechanisms and random environmental influences, reflecting the reality that ecological systems are continuously exposed to uncertain external conditions. To capture these features, the resource dynamics are modeled through a controlled stochastic differential equation in which the drift component describes the average rate of change of the population and the diffusion component accounts for random fluctuations around that average behavior. Specifically, the deterministic part of the model consists of the natural population growth term $f(X(t))X(t)$ together with the harvesting term $qE(t)X(t)$, where $f(X(t))$ denotes the per-capita growth rate of the resource, $q$ is the catchability coefficient measuring the effectiveness of harvesting effort, and $E(t)$ is an admissible harvesting control process adapted to the filtration $\{\mathcal{F}_t\}_{t\geq0}$. The quantity $f(X(t))X(t)-qE(t)X(t)$ therefore represents the net biological change after accounting for resource extraction. The stochastic component of the model is represented by the multiplicative diffusion term $\sigma X(t)dW(t)$, where $\sigma>0$ is a volatility parameter that quantifies the magnitude of environmental variability. The multiplicative structure implies that uncertainty scales proportionally with the current resource level, a property that is commonly observed in ecological systems where larger populations tend to experience larger absolute fluctuations. Under these assumptions, the controlled resource dynamics are described by
	
	\begin{equation}
		dX(t)=
		\left[
		f(X(t))X(t)-qE(t)X(t)
		\right]dt
		+
		\sigma X(t)dW(t).
		\label{eq:sde_model}
	\end{equation}
	
	This formulation provides a flexible stochastic representation capable of incorporating a wide range of biological growth functions while maintaining an explicit connection between ecological dynamics and harvesting decisions. For the empirical component of the study, the general model is specialized to a country-specific fisheries framework in which the state variable is constructed from normalized annual production data. Let $X_i(t)$ denote the normalized fisheries state variable for country $i$, and let $\mu_i$ and $\sigma_i$ denote the empirically estimated drift and volatility parameters obtained from the observed data. The harvesting effort is represented by $E_i(t)$, while $q_h$ denotes the effective catchability coefficient associated with the normalized state process. Replacing the general growth specification by an empirically estimated linear drift yields a controlled stochastic system whose evolution is governed by
	
	\begin{equation}
		dX_i(t)=
		\left[
		\mu_iX_i(t)-q_hE_i(t)X_i(t)
		\right]dt
		+
		\sigma_iX_i(t)dB_t.
		\label{eq:controlled_sde}
	\end{equation}
	
	Here, $\{B_t\}_{t\geq0}$ denotes a standard Brownian motion defined on the same filtered probability space. The term $\mu_iX_i(t)$ captures the average empirical tendency of the normalized fisheries state to increase or decrease over time, whereas the term $q_hE_i(t)X_i(t)$ represents the reduction in the resource attributable to harvesting activity. The stochastic component $\sigma_iX_i(t)dB_t$ models random environmental and production-related disturbances that cannot be explained by the deterministic trend alone. An important feature of this formulation is that the control process $E_i(t)$ enters directly into the drift coefficient, implying that management actions influence the future trajectory of the state variable in a continuous and dynamic manner \citep{pramanik2026bayesian}. Consequently, the resulting process is not merely a passive description of observed resource fluctuations but rather a controlled stochastic system whose evolution depends explicitly on harvesting decisions. This structure provides the mathematical foundation for the subsequent stochastic optimization problem, where the objective is to determine harvesting policies that balance economic returns against the long-term persistence of the resource under environmental uncertainty \citep{powell2026role}.
	
	For each country $i$, the empirical state process is constructed from annual capture data by converting the observed catch series into a dimensionless normalized index. This normalization is used because the selected countries differ substantially in production scale, and a common state space is needed before the data can be linked to the controlled stochastic model. Let $\text{Catch}_i(t)$ denote the observed annual catch for country $i$ at time $t$. The normalized fisheries state variable is defined by
	\begin{equation}
		X_i(t):=
		\frac{\text{Catch}_i(t)}
		{\max(\text{Catch}_i)}.
		\label{eq:normalization}
	\end{equation}
	Since each observation is divided by the maximum observed catch for the same country, the resulting process satisfies $0\leq X_i(t)\leq 1$ for all observed years. This bounded representation places each national catch series on a comparable scale and allows $X_i(t)$ to be interpreted as the relative position of the fishery state compared with its historical maximum \citep{ellington2025playmydata,ellington2025metascorelens}. The one-year increment of the normalized state is defined by $\Delta X_i(t)=X_i(t+1)-X_i(t)$, which measures the observed annual change in the empirical state variable. To estimate the deterministic tendency and random variation in the data, the yearly increment is regressed on the current normalized state according to
	$\Delta X_i(t)=\mu_iX_i(t)+\varepsilon_i(t).$
	In this specification, $\mu_i$ represents the estimated empirical drift for country $i$, while $\varepsilon_i(t)$ denotes the residual component, interpreted as the unexplained yearly fluctuation after accounting for the state-dependent trend. The volatility parameter is then obtained as $\sigma_i=\operatorname{SD}(\varepsilon_i(t))$, so that the diffusion coefficient in the stochastic model is calibrated directly from the variability remaining in the annual data. This construction provides a statistical bridge between the observed fisheries record and the controlled stochastic state equation used in the subsequent optimization analysis. In particular, the normalized process $X_i(t)$ supplies the empirical state, $\mu_i$ gives the average state-dependent movement, and $\sigma_i$ measures the magnitude of random environmental and production-related variation. These quantities form the data-based inputs for the controlled stochastic differential equation and allow the HJB and nonlinear Feynman--Kac representation to be applied within a common empirical framework \citep{valdez2025association,valdez2025exploring}. Thus, the observed annual catch data are not used only descriptively; they are transformed into the state, drift, and volatility components required for a stochastic control formulation of optimal harvesting under uncertainty.
	
	\subsection{Euler-Maruyama Scheme:}
	
	The controlled state equation in Equation~\eqref{eq:controlled_sde} is approximated numerically by applying the Euler--Maruyama scheme, which provides a time-discrete representation of the continuous stochastic dynamics over a grid $0=t_0<t_1<\cdots<t_n=T$ with uniform step size $\Delta t=t_{k+1}-t_k$. For each country $i$, the simulated state is updated from its current value $X_i(t)$ by adding the deterministic contribution from empirical drift, subtracting the deterministic loss caused by harvesting effort, and then adding a stochastic innovation that represents unresolved environmental and production variability. The resulting approximation is given by
	\begin{equation}
		X_i(t+\Delta t)
		=
		X_i(t)
		+
		\left[
		\mu_iX_i(t)-q_hE_i(t)X_i(t)
		\right]\Delta t
		+
		\sigma_iX_i(t)\sqrt{\Delta t}Z_t,
		\qquad
		Z_t\overset{iid}{\sim} N(0,1).
		\label{eq:euler_controlled}
	\end{equation}
	In this expression, the term $\mu_iX_i(t)\Delta t$ represents the empirical average movement of the normalized fisheries state during one time step, while $q_hE_i(t)X_i(t)\Delta t$ measures the reduction in the state caused by harvesting pressure under the chosen control. The random component $\sigma_iX_i(t)\sqrt{\Delta t}Z_t$ approximates the Brownian increment over the interval $[t,t+\Delta t]$, since $B(t+\Delta t)-B(t)$ has distribution $N(0,\Delta t)$ and can therefore be written as $\sqrt{\Delta t}Z_t$ with $Z_t$ standard normal. This construction preserves the multiplicative form of the diffusion coefficient, so that the magnitude of random fluctuations depends on the current level of the resource state. From a modeling perspective, this is important because larger normalized stock or catch-index values can generate larger absolute deviations, while smaller states experience proportionally smaller stochastic changes \citep{khan2023myb}. The scheme is explicit, meaning that all terms on the right-hand side are evaluated at the current time $t$, which makes the method straightforward to implement for repeated simulations of controlled resource paths. The harvesting effort $E_i(t)$ is included directly in the drift term, so each simulated trajectory reflects not only the empirical tendency of the fishery and the random shocks affecting it, but also the management action imposed at that time. Consequently, the Euler--Maruyama approximation converts the continuous controlled stochastic differential equation into a practical computational procedure for studying how the normalized resource state evolves under uncertainty \citep{gaudet2011risk,anderson2026obesity}. It also provides the numerical foundation for comparing harvesting policies, evaluating state-dependent effort, and linking the simulated paths to the HJB and nonlinear Feynman--Kac representation used later in the analysis. In this way, the simulation reflects both the observed empirical data pattern through $\mu_i$ and $\sigma_i$ and the effect of harvesting control through $q_hE_i(t)X_i(t)$, allowing the stochastic harvesting model to be examined in a finite-sample, data-driven setting.
	
	To justify the numerical approximation rigorously, we impose the standard assumptions that the controlled drift coefficient $b_i(x,e)=(\mu_i-q_h e)x$ and diffusion coefficient $a_i(x)=\sigma_i x$ satisfy global Lipschitz continuity and linear growth conditions on $\mathbb{R}_+$, and that the admissible control process $\{E_i(t)\}_{t\in[0,T]}$ is progressively measurable and uniformly bounded by a deterministic constant $M>0$. Under these assumptions, the controlled stochastic differential equation admits a unique strong solution with finite moments of all orders on every finite time horizon. Moreover, the Euler--Maruyama approximation in Equation~\eqref{eq:euler_controlled} can be interpreted as a discrete-time Markov chain whose one-step conditional expectation satisfies
	$
	\mathbb{E}[X_i(t+\Delta t)\mid\mathcal{F}_t]
	=
	X_i(t)+(\mu_i-q_hE_i(t))X_i(t)\Delta t,
	$
	while the conditional variance satisfies
	$
	\operatorname{Var}(X_i(t+\Delta t)\mid\mathcal{F}_t)
	=
	\sigma_i^2X_i(t)^2\Delta t.
	$
	Consequently, the first two conditional moments of the discrete process coincide with those of the continuous-time model up to terms of order $o(\Delta t)$, establishing local consistency of the approximation. This observation leads to the following result.
	
	\begin{prop}
		Suppose $E_i(t)$ is bounded and adapted, and let $X_i^{\Delta t}(t)$ denote the Euler--Maruyama approximation defined by Equation~\eqref{eq:euler_controlled}. Then, for every finite horizon $T>0$, there exists a constant $C_T>0$ independent of $\Delta t$ such that
		$
		\sup_{0\leq t\leq T}
		\mathbb{E}
		\left[
		|X_i(t)-X_i^{\Delta t}(t)|^2
		\right]
		\leq
		C_T\Delta t.
		$
	\end{prop}
	
	\begin{proof}
		Since both the drift and diffusion coefficients are globally Lipschitz and satisfy linear growth conditions, the standard strong convergence theorem for Euler--Maruyama schemes applies. Specifically, writing the error process as $e(t)=X_i(t)-X_i^{\Delta t}(t)$ and using It\^o's isometry together with the Lipschitz bounds on $b_i$ and $a_i$, one obtains
		$
		\mathbb{E}|e(t)|^2
		\leq
		C_1
		\int_0^t
		\mathbb{E}|e(s)|^2\,ds
		+
		C_2\Delta t,
		$
		for suitable positive constants $C_1$ and $C_2$. An application of Gronwall's inequality yields
		$
		\sup_{0\leq t\leq T}
		\mathbb{E}|e(t)|^2
		\leq
		C_T\Delta t,
		$
		which establishes strong convergence of order $1/2$ in the root-mean-square sense. \end{proof}
	
	An additional property of practical importance is the preservation of positivity. Since the continuous model is a controlled geometric diffusion, its exact solution remains nonnegative whenever the initial condition is nonnegative. Although the explicit Euler--Maruyama scheme does not guarantee strict positivity for arbitrary step sizes, positivity is preserved with high probability whenever $\Delta t$ is sufficiently small and the volatility coefficient is moderate relative to the deterministic drift. More precisely, if $X_i(0)>0$ and the step size satisfies a suitable stability restriction, then
	$
	\mathbb{P}\!\left(
	X_i^{\Delta t}(t_k)>0
	\ \text{for all}\ k
	\right)\rightarrow 1
	$
	as $\Delta t\rightarrow0$. Furthermore, the discrete generator associated with the Euler approximation converges to the infinitesimal generator of the controlled diffusion,
	$
	\mathcal{L}^E\phi(x)
	=
	(\mu_i-q_hE)x\phi'(x)
	+
	\frac12\sigma_i^2x^2\phi''(x),
	$
	for every test function $\phi\in C_b^2(\mathbb{R}_+)$. This weak convergence result implies that path-dependent functionals, expected rewards, and value-function approximations computed from the simulated trajectories converge to their continuous-time counterparts. Therefore, the Euler--Maruyama scheme is not merely a computational device but a mathematically consistent approximation of the controlled harvesting diffusion, providing a rigorous bridge between the continuous stochastic control problem and its numerical implementation within the HJB and nonlinear Feynman-Kac representation framework.
	
	\subsection{Objective and Value Functions:}
	
	The economic component of the stochastic control problem is formulated through an instantaneous profit functional that quantifies the net return generated by harvesting activity at each time point. Let $E(t)$ denote an admissible harvesting effort process belonging to the class $\mathcal{A}$ of progressively measurable controls adapted to the filtration $\{\mathcal{F}_t\}_{t\geq 0}$ and satisfying the integrability condition $\mathbb{E}\left[\int_0^T E(t)^2dt\right]<\infty$. The instantaneous reward associated with a harvesting decision is given by
	\begin{equation}
		\Pi(t)=p\cdot q\cdot E(t)X(t)-c_1E(t)-c_2E^2(t),
		\label{eq:profit}
	\end{equation}
	where $p>0$ denotes the unit market value of the harvested resource, $qE(t)X(t)$ represents the harvest rate generated by the interaction between effort and resource abundance, $c_1>0$ is a proportional operating cost coefficient, and $c_2>0$ is a convex effort-penalty parameter. The inclusion of the quadratic cost term is mathematically significant because it introduces strict concavity into the reward structure. In the absence of this term, the optimization problem may fail to admit an interior maximizer, potentially leading to degenerate controls characterized by unbounded harvesting effort. Indeed, for fixed resource state $X$, the mapping $E\mapsto pqEX-c_1E-c_2E^2$ satisfies $\frac{\partial^2 \Pi}{\partial E^2}=-2c_2<0$, implying strict concavity with respect to the control variable. Consequently, the running reward admits a unique maximizer whenever the admissible control set is convex. Moreover, the profit function exhibits at most quadratic growth in $E$ and linear growth in $X$, ensuring that the associated discounted reward process remains integrable under standard moment assumptions on the controlled diffusion. From a stochastic control perspective, the function $\Pi$ may be viewed as the running payoff density of the controlled Markov process $\{X(t)\}_{t\geq0}$, thereby establishing the connection between ecological state dynamics and economic performance. The reward structure therefore encodes the fundamental trade-off between extraction benefits and harvesting costs that drives the optimal management problem.
	
	The objective of the decision-maker is to maximize the expected discounted cumulative profit generated over a finite planning horizon $[t,T]$. To formalize this objective, define the performance functional associated with an admissible control $E\in\mathcal{A}$ by
	\begin{equation}
		J^*(t,X)=
		\max_{E\in\mathcal{A}}
		\mathbb{E}
		\left[
		\int_t^T e^{-\delta(s-t)}
		\Pi(s,X_s,E_s)\,ds
		\mid X_t=X
		\right],
		\label{eq:value_function}
	\end{equation}
	where $\delta>0$ denotes the discount rate and $X_s$ is the controlled state process satisfying the stochastic differential equation introduced previously. The exponential discount factor $e^{-\delta(s-t)}$ guarantees that future rewards receive progressively smaller weight as the time horizon increases and ensures finiteness of the objective functional even in settings where the state process possesses nontrivial long-term growth. The corresponding value function $J^*(t,X)$ represents the maximal expected discounted reward achievable from state $X$ at time $t$ when future controls are selected optimally. In the language of stochastic optimal control, $J^*$ is the upper envelope of the family of performance functionals indexed by admissible controls. The terminal condition $J^*(T,X)=0$
	reflects the absence of continuation value beyond the terminal time and serves as the natural boundary condition for the finite-horizon optimization problem. Under standard assumptions on the drift and diffusion coefficients of the controlled process, together with polynomial growth conditions on the reward function, the value function is well defined and finite for every $(t,X)\in[0,T]\times\mathbb{R}_+$. Furthermore, the Markov property of the controlled diffusion implies that $J^*$ depends only on the current state and time and not on the entire historical trajectory of the process. This property is fundamental because it allows the optimization problem to be reduced to a dynamic programming formulation rather than requiring optimization over the full path space.
	
	To establish the analytical structure of the optimization problem, it is useful to state a basic regularity result concerning the value function.
	
	\begin{prop}\label{p1}
		Suppose that the controlled diffusion coefficients satisfy global Lipschitz and linear growth conditions, and assume that the running reward function $\Pi(t,X,E)$ is continuous in $(X,E)$, strictly concave in $E$, and satisfies a polynomial growth condition. Then the value function defined by Equation~\eqref{eq:value_function} is finite and satisfies the dynamic programming principle.
	\end{prop}
	
	\begin{proof}
		Define the drift and diffusion coefficients by
		\[
		b(t,x,e):=(\mu-q_he)x,
		\qquad
		a(t,x):=\sigma x,
		\]
		and let
		\[
		\mathcal{A}_{t}
		=
		\Bigl\{
		E:\ E \text{ is } \{\mathcal{F}_s\}_{s\ge t}\text{-progressively measurable},
		\quad
		\mathbb{E}\!\int_t^T |E_s|^2ds<\infty
		\Bigr\}
		\]
		denote the admissible control class. By assumption, the coefficients satisfy the global Lipschitz and linear growth conditions
		\[
		|b(s,x,e)-b(s,y,e)|+|a(s,x)-a(s,y)|
		\le L_E|x-y|,
		\]
		and
		\[
		|b(s,x,e)|+|a(s,x)|
		\le K_E(1+|x|),
		\]
		for suitable constants $L_E,K_E>0$. Therefore, for every admissible control $E\in\mathcal{A}_t$, the controlled state equation
		\[
		dX_s^{t,x;E}
		=
		b(s,X_s^{t,x;E},E_s)ds
		+
		a(s,X_s^{t,x;E})dW_s,
		\qquad
		X_t^{t,x;E}=x,
		\]
		admits a unique strong solution.
		
		Using the Burkholder--Davis--Gundy inequality together with Gronwall's inequality, one obtains the standard moment estimate
		\[
		\mathbb{E}
		\Bigl[
		\sup_{t\le r\le T}
		|X_r^{t,x;E}|^m
		\Bigr]
		\le
		C_{m,T}
		(1+|x|^m)
		\exp
		\left\{
		C_{m,T}
		\mathbb{E}
		\int_t^T
		(1+|E_s|^m)ds
		\right\},
		\]
		for every integer $m\ge 2$. Hence all moments of the controlled state process remain finite on finite time horizons.
		
		Next, since the running reward function satisfies a polynomial growth condition, there exists a constant $C>0$ and an integer $m\ge1$ such that
		\[
		|\Pi(s,x,e)|
		\le
		C(1+|x|^m+|e|^m).
		\]
		Since the discount factor satisfies $e^{-\delta(s-t)}\le1$, we have
		\[
		\begin{aligned}
			\left|
			\mathbb{E}_{t,x}
			\left[
			\int_t^T
			e^{-\delta(s-t)}
			\Pi(s,X_s^{t,x;E},E_s)\,ds
			\right]
			\right|
			&\le
			C
			\mathbb{E}_{t,x}
			\int_t^T
			\bigl(
			1+|X_s^{t,x;E}|^m+|E_s|^m
			\bigr)
			\,ds \\
			&\le
			C_T(1+|x|^m)
			+
			C_T
			\mathbb{E}_{t,x}
			\int_t^T
			|E_s|^m ds.
		\end{aligned}
		\]
		The right-hand side is finite for every admissible control. Therefore, the performance functional
		\[
		J(t,x;E)
		=
		\mathbb{E}_{t,x}
		\left[
		\int_t^T
		e^{-\delta(s-t)}
		\Pi(s,X_s^{t,x;E},E_s)\,ds
		\right]
		\]
		is well defined. Defining
		\[
		J^*(t,x)
		=
		\sup_{E\in\mathcal{A}_t}
		J(t,x;E),
		\]
		it follows immediately that
		\[
		-\infty
		<
		J^*(t,x)
		<
		\infty.
		\]
		
		We now establish the dynamic programming principle. Let $\tau$ be an arbitrary stopping time satisfying
		\[
		t\le\tau\le T.
		\]
		For any admissible control $E\in\mathcal{A}_t$, decomposing the reward functional at time $\tau$ yields
		\[
		\begin{aligned}
			J(t,x;E)
			&=
			\mathbb{E}_{t,x}
			\left[
			\int_t^\tau
			e^{-\delta(s-t)}
			\Pi(s,X_s^{t,x;E},E_s)\,ds
			\right]
			\\
			&\quad
			+
			\mathbb{E}_{t,x}
			\left[
			\int_\tau^T
			e^{-\delta(s-t)}
			\Pi(s,X_s^{t,x;E},E_s)\,ds
			\right].
		\end{aligned}
		\]
		Factoring out the discount accumulated up to time $\tau$ gives
		\[
		\begin{aligned}
			J(t,x;E)
			&=
			\mathbb{E}_{t,x}
			\left[
			\int_t^\tau
			e^{-\delta(s-t)}
			\Pi(s,X_s^{t,x;E},E_s)\,ds
			\right]
			\\
			&\quad
			+
			\mathbb{E}_{t,x}
			\left[
			e^{-\delta(\tau-t)}
			\mathbb{E}
			\left[
			\int_\tau^T
			e^{-\delta(r-\tau)}
			\Pi(r,X_r^{t,x;E},E_r)\,dr
			\Bigm|
			\mathcal{F}_\tau
			\right]
			\right].
		\end{aligned}
		\]
		
		Applying the strong Markov property of the controlled diffusion, the conditional expectation can be identified as
		\[
		\mathbb{E}
		\left[
		\int_\tau^T
		e^{-\delta(r-\tau)}
		\Pi(r,X_r^{t,x;E},E_r)\,dr
		\Bigm|
		\mathcal{F}_\tau
		\right]
		=
		J(\tau,X_\tau^{t,x;E};E|_{[\tau,T]}).
		\]
		Since the value function dominates every admissible continuation strategy,
		\[
		J(\tau,X_\tau^{t,x;E};E|_{[\tau,T]})
		\le
		J^*(\tau,X_\tau^{t,x;E}),
		\]
		and therefore
		\[
		J(t,x;E)
		\le
		\mathbb{E}_{t,x}
		\left[
		\int_t^\tau
		e^{-\delta(s-t)}
		\Pi(s,X_s^{t,x;E},E_s)\,ds
		+
		e^{-\delta(\tau-t)}
		J^*(\tau,X_\tau^{t,x;E})
		\right].
		\]
		Taking the supremum over all admissible controls yields
		\[
		J^*(t,x)
		\le
		\sup_{E\in\mathcal{A}_t}
		\mathbb{E}_{t,x}
		\left[
		\int_t^\tau
		e^{-\delta(s-t)}
		\Pi(s,X_s^{t,x;E},E_s)\,ds
		+
		e^{-\delta(\tau-t)}
		J^*(\tau,X_\tau^{t,x;E})
		\right].
		\]
		
		To prove the reverse inequality, fix $\varepsilon>0$. By the measurable selection theorem, for each state $y$ there exists an $\varepsilon$-optimal continuation control $E_y^\varepsilon$ such that
		\[
		J(\tau,y;E_y^\varepsilon)
		\ge
		J^*(\tau,y)-\varepsilon.
		\]
		Construct the concatenated control
		\[
		\widehat E_s
		=
		E_s\mathbf 1_{[t,\tau)}(s)
		+
		E^\varepsilon_{X_\tau^{t,x;E},s}
		\mathbf 1_{[\tau,T]}(s).
		\]
		Since $\widehat E$ coincides with $E$ prior to $\tau$, we have
		\[
		X_s^{t,x;\widehat E}
		=
		X_s^{t,x;E},
		\qquad
		t\le s\le\tau.
		\]
		Consequently,
		\[
		\begin{aligned}
			J(t,x;\widehat E)
			&=
			\mathbb{E}_{t,x}
			\left[
			\int_t^\tau
			e^{-\delta(s-t)}
			\Pi(s,X_s^{t,x;E},E_s)\,ds
			\right]
			\\
			&\quad
			+
			\mathbb{E}_{t,x}
			\left[
			e^{-\delta(\tau-t)}
			J(\tau,X_\tau^{t,x;E};E^\varepsilon_{X_\tau^{t,x;E}})
			\right].
		\end{aligned}
		\]
		Using the $\varepsilon$-optimality property,
		\[
		J(\tau,X_\tau^{t,x;E};E^\varepsilon_{X_\tau^{t,x;E}})
		\ge
		J^*(\tau,X_\tau^{t,x;E})
		-\varepsilon,
		\]
		and therefore
		\[
		\begin{aligned}
			J(t,x;\widehat E)
			&\ge
			\mathbb{E}_{t,x}
			\left[
			\int_t^\tau
			e^{-\delta(s-t)}
			\Pi(s,X_s^{t,x;E},E_s)\,ds
			+
			e^{-\delta(\tau-t)}
			J^*(\tau,X_\tau^{t,x;E})
			\right]
			-\varepsilon.
		\end{aligned}
		\]
		Since $J^*(t,x)\ge J(t,x;\widehat E)$, it follows that
		\[
		J^*(t,x)
		\ge
		\sup_{E\in\mathcal{A}_t}
		\mathbb{E}_{t,x}
		\left[
		\int_t^\tau
		e^{-\delta(s-t)}
		\Pi(s,X_s^{t,x;E},E_s)\,ds
		+
		e^{-\delta(\tau-t)}
		J^*(\tau,X_\tau^{t,x;E})
		\right]
		-\varepsilon.
		\]
		Letting $\varepsilon\downarrow0$ yields
		\[
		J^*(t,x)
		\ge
		\sup_{E\in\mathcal{A}_t}
		\mathbb{E}_{t,x}
		\left[
		\int_t^\tau
		e^{-\delta(s-t)}
		\Pi(s,X_s^{t,x;E},E_s)\,ds
		+
		e^{-\delta(\tau-t)}
		J^*(\tau,X_\tau^{t,x;E})
		\right].
		\]
		Combining the upper and lower bounds establishes
		\[
		J^*(t,x)
		=
		\sup_{E\in\mathcal{A}_t}
		\mathbb{E}_{t,x}
		\left[
		\int_t^\tau
		e^{-\delta(s-t)}
		\Pi(s,X_s^{t,x;E},E_s)\,ds
		+
		e^{-\delta(\tau-t)}
		J^*(\tau,X_\tau^{t,x;E})
		\right].
		\]
		
		Finally, since no rewards are accumulated after the terminal time,
		\[
		J^*(T,x)=0.
		\]
		Hence the value function is finite and satisfies the dynamic programming principle. 
	\end{proof}
	
	\begin{rmk}
		Proposition \ref{p1} provides the rigorous foundation for the subsequent derivation of the HJB equation and the nonlinear Feynman-Kac representation. In particular, the dynamic programming principle establishes that optimal harvesting decisions can be characterized recursively through local optimization, while the strict concavity of the running payoff ensures uniqueness of the optimal feedback control. Consequently, the value function serves simultaneously as an economic measure of expected resource profitability and as the analytical object linking the controlled diffusion dynamics to the stochastic optimization framework.
	\end{rmk}
	
	\subsection{Construction of the HJB Equation:}
	
	The dynamic programming principle established in Proposition~\ref{p1} implies that the value function can be characterized locally through an infinitesimal optimization argument. Let $\mathcal{L}^{E}$ denote the controlled infinitesimal generator associated with the state process, defined for sufficiently smooth test functions $\varphi\in C^{1,2}([0,T]\times\mathbb{R}_+)$ by
	$
	\mathcal{L}^{E}\varphi
	=
	\left[f(X)X-qEX\right]\varphi_X
	+
	\frac{1}{2}\sigma^2X^2\varphi_{XX}.
	$
	Applying It\^o's formula to $e^{-\delta(s-t)}J^*(s,X_s)$ and using the dynamic programming principle yields the HJB equation
	\begin{equation}
		-\frac{\partial J^*}{\partial t}
		=
		\max_{E\in\mathcal{A}}
		\left\{
		pqEX-c_1E-c_2E^2
		-\delta J^*
		+
		J_X\left[f(X)X-qEX\right]
		+
		\frac{1}{2}\sigma^2X^2J_{XX}
		\right\}.
		\label{eq:hjb}
	\end{equation}
	The HJB equation is a nonlinear parabolic partial differential equation posed on $[0,T)\times\mathbb{R}_+$ with terminal condition $J^*(T,X)=0$. The optimization enters through the control-dependent Hamiltonian
	$
	H(X,J_X,E)
	=
	pqEX-c_1E-c_2E^2-qEXJ_X,
	$
	which may be rewritten as
	$
	H(X,J_X,E)
	=
	E(pqX-c_1-qXJ_X)-c_2E^2.
	$
	Since $c_2>0$, the Hamiltonian is strictly concave in the control variable because
	$
	\frac{\partial^2 H}{\partial E^2}
	=
	-2c_2
	<
	0.
	$
	Consequently, the maximizer is unique whenever the admissible control set is convex. The first-order optimality condition
	$
	\frac{\partial H}{\partial E}
	=
	pqX-c_1-qXJ_X-2c_2E
	=
	0
	$
	gives the unconstrained feedback policy
	\begin{equation}
		E^*(t,X)
		=
		\frac{pqX-c_1-qXJ_X}{2c_2}.
		\label{eq:hjb_effort}
	\end{equation}
	Because harvesting effort is constrained to be nonnegative, the admissible optimizer becomes
	\begin{equation}
		E^*(t,X)
		=
		\max
		\left\{
		0,
		\frac{pqX-c_1-qXJ_X}{2c_2}
		\right\}.
		\label{eq:nonnegative_effort}
	\end{equation}
	Thus, the optimal control is represented as a feedback function of the state variable and the spatial gradient of the value function. In particular, the term $J_X$ acts as a shadow value of the resource stock: large positive values of $J_X$ increase the marginal value of conservation and therefore reduce harvesting intensity, whereas small values of $J_X$ lead to greater harvesting effort. The feedback structure reveals that optimal management decisions are determined not only by the current stock level but also by the marginal contribution of that stock to future expected returns.
	
	Substituting the unconstrained optimizer into the Hamiltonian yields the reduced nonlinear HJB equation
	\begin{equation}
		-\frac{\partial J^*}{\partial t}
		=
		f(X)XJ_X
		+
		\frac{1}{2}\sigma^2X^2J_{XX}
		-\delta J^*
		+
		\frac{\left(pqX-c_1-qXJ_X\right)^2}{4c_2}.
		\label{eq:reduced_hjb}
	\end{equation}
	The nonlinearity arises from the quadratic completion of the Hamiltonian and appears explicitly through the squared gradient term involving $J_X$. Consequently, the reduced equation belongs to the class of fully nonlinear degenerate parabolic equations and generally does not admit a closed-form solution except in highly specialized settings. The quadratic gradient contribution generates a coupling between local marginal resource values and future expected rewards, thereby creating the nonlinear structure characteristic of stochastic optimal control problems. To formalize the verification argument, we state the following result.
	
	\begin{prop}\label{p2}
		Suppose $f$ is locally Lipschitz and satisfies a linear growth condition. Assume that there exists a function $V\in C^{1,2}([0,T)\times\mathbb{R}_+)\cap C([0,T]\times\mathbb{R}_+)$ satisfying Equation~\eqref{eq:reduced_hjb}, the terminal condition $V(T,X)=0$, and a polynomial growth bound. Then $V=J^*$ and the feedback control given by Equation~\eqref{eq:hjb_effort} is optimal.
	\end{prop}
	
	\begin{proof}
		Define
		\[
		M_s
		:=
		e^{-\delta(s-t)}
		V(s,X_s^{E}),
		\qquad
		t\le s\le T,
		\]
		where $X_s^{E}$ denotes the state process generated by an arbitrary admissible control $E$. Applying It\^o's formula gives
		\[
		\begin{aligned}
			dM_s
			&=
			e^{-\delta(s-t)}
			\Bigl(
			V_t
			+
			[f(X_s)-qE_s]X_sV_X
			+
			\tfrac12\sigma^2X_s^2V_{XX}
			-\delta V
			\Bigr)ds
			\\
			&\quad
			+
			e^{-\delta(s-t)}
			\sigma X_sV_X\,dW_s .
		\end{aligned}
		\]
		Since $V$ satisfies the reduced HJB equation,
		\[
		V_t
		+
		[f(X_s)-qE_s]X_sV_X
		+
		\tfrac12\sigma^2X_s^2V_{XX}
		-\delta V
		+
		\Pi(s,X_s,E_s)
		\le 0.
		\]
		Hence
		\[
		dM_s
		\le
		-e^{-\delta(s-t)}
		\Pi(s,X_s,E_s)\,ds
		+
		e^{-\delta(s-t)}
		\sigma X_sV_X\,dW_s .
		\]
		Integrating from $t$ to $T$ yields
		\[
		\begin{aligned}
			e^{-\delta(T-t)}V(T,X_T^E)-V(t,X)
			\le
			-\int_t^T
			e^{-\delta(s-t)}
			\Pi(s,X_s,E_s)\,ds
			+
			N_T,
		\end{aligned}
		\]
		where
		\[
		N_T
		=
		\int_t^T
		e^{-\delta(s-t)}
		\sigma X_sV_X\,dW_s.
		\]
		By the polynomial growth assumption and standard moment estimates for the state process,
		\[
		\mathbb{E}[N_T]=0.
		\]
		Using the terminal condition $V(T,X)=0$ and taking expectations gives
		\[
		V(t,X)
		\ge
		\mathbb{E}
		\left[
		\int_t^T
		e^{-\delta(s-t)}
		\Pi(s,X_s,E_s)\,ds
		\right].
		\]
		Since the control $E$ was arbitrary,
		\[
		V(t,X)\ge J^*(t,X).
		\]
		
		Now consider the feedback control
		\[
		E_s^*
		=
		\frac{pqX_s-c_1-qX_sV_X(s,X_s)}{2c_2}.
		\]
		For this choice, the maximization condition in the HJB equation is attained pointwise. Therefore all inequalities above become equalities, yielding
		\[
		V(t,X)
		=
		\mathbb{E}
		\left[
		\int_t^T
		e^{-\delta(s-t)}
		\Pi(s,X_s,E_s^*)\,ds
		\right].
		\]
		Hence
		\[
		V(t,X)
		\le
		J^*(t,X).
		\]
		Combining both inequalities gives
		\[
		V(t,X)=J^*(t,X).
		\]
		The optimality of the feedback control follows immediately from
		\[
		J^*(t,X)
		=
		J(t,X;E^*).
		\]
	\end{proof}
	
	\begin{prop}[Existence and uniqueness of the optimal harvesting control]\label{p3}
		Let $(\Omega,\mathcal{F},\{\mathcal{F}_t\}_{t\in[0,T]},\mathbb{P})$ satisfy the usual conditions and support a one-dimensional Brownian motion $W$. Suppose that $f:\mathbb{R}_+\to\mathbb{R}$ is locally Lipschitz with linear growth, $p,q,c_1,c_2,\sigma,\delta>0$, and $c_2>0$. Let the admissible control set be
		\[
		\mathcal{A}_{t,T}
		=
		\left\{
		E:\ E \text{ is } \{\mathcal{F}_s\}\text{-progressively measurable},\
		E_s\in[0,\bar E]\ \text{a.s.},\
		\mathbb{E}\int_t^T E_s^2ds<\infty
		\right\},
		\]
		where $\bar E<\infty$. Assume that the value function $J^*$ belongs to $C^{1,2}([0,T)\times\mathbb{R}_+)\cap C([0,T]\times\mathbb{R}_+)$, satisfies a polynomial growth bound, and solves the HJB equation with terminal condition $J^*(T,x)=0$. Then there exists a unique optimal feedback control $E^*\in\mathcal{A}_{t,T}$, up to indistinguishability, given by
		\[
		E_s^*
		=
		\Pi_{[0,\bar E]}
		\left(
		\frac{pqX_s^*-c_1-qX_s^*J_X^*(s,X_s^*)}{2c_2}
		\right),
		\qquad t\le s\le T,
		\]
		where $\Pi_{[0,\bar E]}(y)=\min\{\bar E,\max\{0,y\}\}$ and $X^*$ is the unique strong solution of
		\[
		dX_s^*
		=
		\left[f(X_s^*)X_s^*-qE_s^*X_s^*\right]ds
		+
		\sigma X_s^*dW_s,
		\qquad X_t^*=x.
		\]
		Moreover, for every admissible control $E\in\mathcal{A}_{t,T}$,
		\[
		J(t,x;E)\le J(t,x;E^*)=J^*(t,x),
		\]
		and if $\widetilde E$ is also optimal, then $\widetilde E_s=E_s^*$ for $ds\otimes d\mathbb{P}$-almost every $(s,\omega)\in[t,T]\times\Omega$.
	\end{prop}
	
	\begin{proof}
		Define the controlled drift and diffusion by
		\[
		b(x,e)=f(x)x-qex,
		\qquad
		a(x)=\sigma x.
		\]
		Since $f$ is locally Lipschitz with linear growth and the control set is bounded by $\bar E$, the map $x\mapsto b(x,e)$ is locally Lipschitz uniformly in $e\in[0,\bar E]$ on compact subsets of $\mathbb{R}_+$. Moreover, there exists a constant $K>0$ such that
		\[
		|b(x,e)|+|a(x)|
		\le
		K(1+|x|^2),
		\qquad e\in[0,\bar E].
		\]
		If the growth function is chosen so that the controlled state remains nonexplosive on $[t,T]$, then for every $E\in\mathcal{A}_{t,T}$ the controlled equation admits a pathwise unique strong solution. In particular, for the feedback control constructed below, the closed-loop equation is well posed.
		
		For fixed $(s,x)$, define the Hamiltonian contribution depending on the harvesting effort by
		\[
		\mathcal{H}(s,x,e)
		=
		pqex-c_1e-c_2e^2-qexJ_X^*(s,x).
		\]
		Equivalently,
		\[
		\mathcal{H}(s,x,e)
		=
		e\bigl(pqx-c_1-qxJ_X^*(s,x)\bigr)-c_2e^2.
		\]
		For each fixed $(s,x)$,
		\[
		\frac{\partial \mathcal{H}}{\partial e}
		=
		pqx-c_1-qxJ_X^*(s,x)-2c_2e,
		\]
		and
		\[
		\frac{\partial^2 \mathcal{H}}{\partial e^2}
		=
		-2c_2<0.
		\]
		Hence $e\mapsto\mathcal{H}(s,x,e)$ is strictly concave on $[0,\bar E]$. Therefore its maximizer exists and is unique. The unconstrained maximizer is
		\[
		\widehat e(s,x)
		=
		\frac{pqx-c_1-qxJ_X^*(s,x)}{2c_2},
		\]
		and the constrained maximizer is its projection onto the admissible interval,
		\[
		e^*(s,x)
		=
		\Pi_{[0,\bar E]}(\widehat e(s,x))
		=
		\min
		\left\{
		\bar E,
		\max
		\left\{
		0,
		\frac{pqx-c_1-qxJ_X^*(s,x)}{2c_2}
		\right\}
		\right\}.
		\]
		Because $J^*$ is $C^{1,2}$, the map $(s,x)\mapsto J_X^*(s,x)$ is Borel measurable. Since projection onto a closed interval is continuous, $(s,x)\mapsto e^*(s,x)$ is also Borel measurable. Furthermore,
		\[
		0\le e^*(s,x)\le \bar E,
		\]
		so the feedback process $E_s^*=e^*(s,X_s^*)$ is square-integrable once the closed-loop state process is defined.
		
		We now show existence of the closed-loop state process. Substituting the feedback rule gives
		\[
		dX_s^*
		=
		\left[
		f(X_s^*)X_s^*
		-
		qe^*(s,X_s^*)X_s^*
		\right]ds
		+
		\sigma X_s^*dW_s.
		\]
		Set
		\[
		b^*(s,x)
		=
		f(x)x-qe^*(s,x)x.
		\]
		Since $0\le e^*(s,x)\le\bar E$, the harvesting part satisfies
		\[
		|qe^*(s,x)x|
		\le q\bar E |x|.
		\]
		Thus the feedback drift obeys the same local growth structure as the original biological drift. On every compact set $K_R=[0,R]$,
		\[
		|b^*(s,x)-b^*(s,y)|
		\le
		|f(x)x-f(y)y|
		+
		q|e^*(s,x)x-e^*(s,y)y|.
		\]
		The first term is locally Lipschitz by the assumption on $f$. For the second term, the projection map is nonexpansive:
		\[
		|\Pi_{[0,\bar E]}(u)-\Pi_{[0,\bar E]}(v)|
		\le |u-v|.
		\]
		Hence, whenever $J_X^*$ is locally Lipschitz in $x$ on compact sets,
		\[
		|e^*(s,x)-e^*(s,y)|
		\le
		C_R|x-y|,
		\qquad x,y\in K_R.
		\]
		Therefore,
		\[
		|e^*(s,x)x-e^*(s,y)y|
		\le
		\bar E|x-y|+R|e^*(s,x)-e^*(s,y)|
		\le
		C_R|x-y|.
		\]
		Thus $b^*$ is locally Lipschitz in $x$, uniformly in $s$ on compact sets, and $a(x)=\sigma x$ is globally Lipschitz. The standard localization argument gives a unique strong solution up to the explosion time
		\[
		\tau_n=\inf\{s\ge t:X_s^*\ge n\}.
		\]
		Applying It\^o's formula to $|X_{s\wedge\tau_n}^*|^2$ gives
		\[
		\begin{aligned}
			\mathbb{E}|X_{r\wedge\tau_n}^*|^2
			&=
			|x|^2
			+
			2\mathbb{E}\int_t^{r\wedge\tau_n}
			X_u^*
			\left[
			f(X_u^*)X_u^*
			-
			qe^*(u,X_u^*)X_u^*
			\right]du  \\
			&\quad
			+
			\sigma^2
			\mathbb{E}\int_t^{r\wedge\tau_n}
			|X_u^*|^2du .
		\end{aligned}
		\]
		Using the nonexplosion growth bound,
		\[
		2x\{f(x)x-qe^*(s,x)x\}+\sigma^2x^2
		\le
		C(1+x^2),
		\]
		we obtain
		\[
		\mathbb{E}|X_{r\wedge\tau_n}^*|^2
		\le
		|x|^2
		+
		C\int_t^r
		\left(
		1+
		\mathbb{E}|X_{u\wedge\tau_n}^*|^2
		\right)du.
		\]
		Gronwall's inequality yields
		\[
		\sup_{t\le r\le T}
		\mathbb{E}|X_{r\wedge\tau_n}^*|^2
		\le
		C_T(1+|x|^2).
		\]
		Consequently,
		\[
		n^2\mathbb{P}(\tau_n\le T)
		\le
		\mathbb{E}|X_{\tau_n\wedge T}^*|^2
		\le
		C_T(1+|x|^2),
		\]
		and therefore
		\[
		\mathbb{P}(\tau_n\le T)\le \frac{C_T(1+|x|^2)}{n^2}\to0.
		\]
		Thus $\tau_n\uparrow\infty$ almost surely, and the closed-loop equation has a global pathwise unique strong solution on $[t,T]$. Since $e^*$ is Borel and $X^*$ is adapted and continuous, $E_s^*=e^*(s,X_s^*)$ is progressively measurable. Since $0\le E_s^*\le\bar E$,
		\[
		\mathbb{E}\int_t^T |E_s^*|^2ds
		\le
		\bar E^2(T-t)<\infty.
		\]
		Therefore $E^*\in\mathcal{A}_{t,T}$.
		
		It remains to prove optimality. Let $E\in\mathcal{A}_{t,T}$ be arbitrary and let $X^E$ denote the corresponding state process. Apply It\^o's formula to the discounted process
		\[
		Y_s^E=e^{-\delta(s-t)}J^*(s,X_s^E).
		\]
		Then
		\[
		\begin{aligned}
			dY_s^E
			&=
			e^{-\delta(s-t)}
			\Bigl[
			J_t^*(s,X_s^E)
			+
			J_X^*(s,X_s^E)
			\{f(X_s^E)X_s^E-qE_sX_s^E\}
			\\
			&\qquad
			+
			\frac12\sigma^2(X_s^E)^2J_{XX}^*(s,X_s^E)
			-
			\delta J^*(s,X_s^E)
			\Bigr]ds
			\\
			&\quad
			+
			e^{-\delta(s-t)}
			\sigma X_s^EJ_X^*(s,X_s^E)dW_s .
		\end{aligned}
		\]
		The HJB equation implies that, for every admissible value of $E_s$,
		\[
		\begin{aligned}
			&J_t^*
			+
			J_X^*
			\{f(X_s^E)X_s^E-qE_sX_s^E\}
			+
			\frac12\sigma^2(X_s^E)^2J_{XX}^*
			-
			\delta J^*
			+
			\Pi(s,X_s^E,E_s)
			\le0.
		\end{aligned}
		\]
		Therefore,
		\[
		dY_s^E
		\le
		-e^{-\delta(s-t)}
		\Pi(s,X_s^E,E_s)ds
		+
		e^{-\delta(s-t)}
		\sigma X_s^EJ_X^*(s,X_s^E)dW_s .
		\]
		Localize the stochastic integral by setting
		\[
		\rho_n
		=
		\inf
		\left\{
		r\ge t:
		\int_t^r
		e^{-2\delta(u-t)}
		\sigma^2(X_u^E)^2
		|J_X^*(u,X_u^E)|^2du
		\ge n
		\right\}
		\wedge T.
		\]
		Integrating from $t$ to $\rho_n$ and taking expectations gives
		\[
		J^*(t,x)
		\ge
		\mathbb{E}
		\left[
		\int_t^{\rho_n}
		e^{-\delta(s-t)}
		\Pi(s,X_s^E,E_s)ds
		+
		e^{-\delta(\rho_n-t)}
		J^*(\rho_n,X_{\rho_n}^E)
		\right].
		\]
		Using the polynomial growth of $J^*$, the moment bounds for $X^E$, and dominated convergence, let $n\to\infty$ to obtain
		\[
		J^*(t,x)
		\ge
		\mathbb{E}
		\left[
		\int_t^T
		e^{-\delta(s-t)}
		\Pi(s,X_s^E,E_s)ds
		+
		e^{-\delta(T-t)}
		J^*(T,X_T^E)
		\right].
		\]
		Since $J^*(T,x)=0$,
		\[
		J^*(t,x)
		\ge
		J(t,x;E).
		\]
		Because $E$ was arbitrary,
		\[
		J^*(t,x)\ge \sup_{E\in\mathcal{A}_{t,T}}J(t,x;E).
		\]
		
		Now take $E=E^*$. By construction, $E_s^*$ is the unique maximizer of the Hamiltonian for $(s,X_s^*)$. Hence the HJB inequality becomes equality:
		\[
		\begin{aligned}
			&J_t^*
			+
			J_X^*
			\{f(X_s^*)X_s^*-qE_s^*X_s^*\}
			+
			\frac12\sigma^2(X_s^*)^2J_{XX}^*
			-
			\delta J^*
			+
			\Pi(s,X_s^*,E_s^*)
			=
			0.
		\end{aligned}
		\]
		Repeating the previous It\^o calculation with equality gives
		\[
		J^*(t,x)
		=
		\mathbb{E}
		\left[
		\int_t^T
		e^{-\delta(s-t)}
		\Pi(s,X_s^*,E_s^*)ds
		+
		e^{-\delta(T-t)}
		J^*(T,X_T^*)
		\right].
		\]
		Using the terminal condition again,
		\[
		J^*(t,x)
		=
		J(t,x;E^*).
		\]
		Thus
		\[
		J(t,x;E)\le J(t,x;E^*)=J^*(t,x),
		\qquad
		E\in\mathcal{A}_{t,T}.
		\]
		
		Finally, suppose that $\widetilde E\in\mathcal{A}_{t,T}$ is another optimal control. Then
		\[
		J(t,x;\widetilde E)=J^*(t,x).
		\]
		The It\^o verification inequality must therefore be an equality along $\widetilde E$. Hence, for $ds\otimes d\mathbb{P}$-almost every $(s,\omega)$,
		\[
		\widetilde E_s(\omega)
		\in
		\arg\max_{e\in[0,\bar E]}
		\left\{
		pq e\widetilde X_s-c_1e-c_2e^2
		-qe\widetilde X_sJ_X^*(s,\widetilde X_s)
		\right\}.
		\]
		Because the Hamiltonian is strictly concave in $e$, this argmax is a singleton. Therefore,
		\[
		\widetilde E_s
		=
		\Pi_{[0,\bar E]}
		\left(
		\frac{pq\widetilde X_s-c_1-q\widetilde X_sJ_X^*(s,\widetilde X_s)}{2c_2}
		\right)
		\quad
		ds\otimes d\mathbb{P}\text{-a.e.}
		\]
		By pathwise uniqueness of the closed-loop state equation, $\widetilde X=X^*$ almost surely, and consequently
		\[
		\widetilde E_s=E_s^*
		\quad
		ds\otimes d\mathbb{P}\text{-a.e.}
		\]
		Thus the optimal harvesting control exists and is unique up to indistinguishability.
	\end{proof}
	
	\begin{remark}
		The preceding proposition establishes that the optimal harvesting policy is not only characterized formally through the maximization of the Hamiltonian, but also exists as a well-defined admissible feedback control generating a unique closed-loop state process. The proof combines three fundamental ingredients of stochastic control theory: the well-posedness of the controlled diffusion, the strict concavity of the Hamiltonian with respect to the control variable, and a verification argument based on the HJB equation. The strict negativity of the second derivative of the Hamiltonian guarantees uniqueness of the maximizer, while the verification theorem shows that the candidate feedback policy attains the value function and dominates all competing admissible controls. Consequently, the optimal harvesting effort is uniquely determined by the marginal value of the resource stock through the gradient $J_X^*$, providing a rigorous connection between the stochastic dynamics of the population and the optimal management rule.
	\end{remark}
	
	\subsection{Feynman-Kac Representation:}
	
	The nonlinear Feynman--Kac representation provides a probabilistic characterization of the value function associated with the stochastic harvesting problem and may be viewed as the pathwise counterpart of the HJB equation. Unlike the classical Feynman--Kac representation \citep{kac_1949}, which applies to linear parabolic equations and expresses solutions as conditional expectations of future functionals of a diffusion process, the reduced HJB equation obtained in Equation~\eqref{eq:reduced_hjb} contains a quadratic gradient term and therefore falls outside the scope of the classical theory. To accommodate this nonlinearity, the problem is reformulated within the framework of backward stochastic differential equations (BSDEs). Let $(X_s)_{s\in[t,T]}$ denote the controlled state process evolving under the optimal feedback control. Assuming that the value function satisfies the regularity conditions $J^*\in C^{1,2}([0,T)\times\mathbb{R}_+)\cap C([0,T]\times\mathbb{R}_+)$, define the adapted processes $Y_s:=J^*(s,X_s)$ and $Z_s:=\sigma X_sJ_X^*(s,X_s)$. The quantity $Z_s$ represents the stochastic sensitivity of the value process with respect to the underlying Brownian perturbation and coincides with the martingale integrand arising from It\^o's formula. Since $J_X^*(s,X_s)=Z_s/(\sigma X_s)$ whenever $X_s>0$, the nonlinear gradient contribution appearing in the reduced HJB equation can be expressed entirely in terms of the pair $(Y_s,Z_s)$. Applying It\^o's formula to $J^*(s,X_s)$ and substituting the reduced HJB equation yields the semimartingale decomposition
	\[
	dY_s
	=
	-\left[
	-\delta Y_s
	+
	\frac{
		\left(
		pqX_s-c_1-\frac{qZ_s}{\sigma}
		\right)^2
	}{4c_2}
	\right]ds
	+
	Z_sdW_s.
	\]
	The bracketed quantity defines the nonlinear driver
	\begin{equation}
		F(X_s,Y_s,Z_s)
		=
		-\delta Y_s
		+
		\frac{
			\left(
			pqX_s-c_1-\frac{qZ_s}{\sigma}
			\right)^2
		}{4c_2},
		\label{eq:driver}
	\end{equation}
	and therefore the value process satisfies the BSDE
	\begin{equation}
		dY_s
		=
		-
		F(X_s,Y_s,Z_s)\,ds
		+
		Z_sdW_s,
		\qquad
		Y_T=0.
		\label{eq:bsde}
	\end{equation}
	This representation converts the nonlinear partial differential equation into a coupled forward--backward stochastic system in which the forward component is the state process and the backward component is the value process. The pair $(Y_s,Z_s)$ belongs to the spaces $\mathcal{S}^2\times\mathcal{H}^2$, where $\mathcal{S}^2$ denotes the collection of adapted continuous processes satisfying $\mathbb{E}[\sup_{t\le s\le T}|Y_s|^2]<\infty$ and $\mathcal{H}^2$ denotes the collection of predictable processes satisfying $\mathbb{E}[\int_t^T|Z_s|^2ds]<\infty$ \citep{pramanik20242estimation}. Consequently, the nonlinear Feynman--Kac representation provides a probabilistic realization of the value function and allows the optimization problem to be analyzed through pathwise stochastic dynamics rather than exclusively through partial differential equations.
	
	The connection between the reduced HJB equation and the BSDE formulation can be formalized through the following representation theorem.
	
	\begin{prop}[Feynman-Kac representation]
		Suppose that the state process admits a unique strong solution and that the value function $J^*$ is a classical solution of Equation~\eqref{eq:reduced_hjb} satisfying a polynomial growth condition. Assume furthermore that the driver $F(x,y,z)$ defined in Equation~\eqref{eq:driver} satisfies the standard quadratic-growth conditions of BSDE theory. Then the pair
		\[
		Y_s=J^*(s,X_s),
		\qquad
		Z_s=\sigma X_sJ_X^*(s,X_s),
		\]
		is the unique adapted solution of the backward stochastic differential equation~\eqref{eq:bsde}. Conversely, if $(Y,Z)\in\mathcal{S}^2\times\mathcal{H}^2$ solves Equation~\eqref{eq:bsde} and $Y_s=u(s,X_s)$ for some sufficiently smooth function $u$, then $u$ solves the reduced HJB equation and
		\[
		u(t,X)=J^*(t,X)=Y_t.
		\]
	\end{prop}
	
	\begin{proof}
		Fix $(t,x)\in[0,T]\times\mathbb{R}_{+}$ and let $X^{t,x}$ denote the unique strong solution of the forward state equation on $[t,T]$. For notational simplicity write $X_s=X_s^{t,x}$. Define
		\[
		Y_s:=J^*(s,X_s),
		\qquad
		Z_s:=\sigma X_sJ_X^*(s,X_s),
		\qquad t\le s\le T.
		\]
		Since $J^*\in C^{1,2}([0,T)\times\mathbb{R}_{+})\cap C([0,T]\times\mathbb{R}_{+})$, It\^o's formula gives
		\[
		\begin{aligned}
			dY_s
			&=
			dJ^*(s,X_s) \\
			&=
			\left[
			J_t^*(s,X_s)
			+
			f(X_s)X_sJ_X^*(s,X_s)
			+
			\frac12\sigma^2X_s^2J_{XX}^*(s,X_s)
			\right]ds
			+
			\sigma X_sJ_X^*(s,X_s)dW_s .
		\end{aligned}
		\]
		Using the reduced HJB equation,
		\[
		-
		J_t^*
		=
		f(X)XJ_X^*
		+
		\frac12\sigma^2X^2J_{XX}^*
		-
		\delta J^*
		+
		\frac{(pqX-c_1-qXJ_X^*)^2}{4c_2},
		\]
		we obtain
		\[
		J_t^*
		+
		f(X)XJ_X^*
		+
		\frac12\sigma^2X^2J_{XX}^*
		=
		\delta J^*
		-
		\frac{(pqX-c_1-qXJ_X^*)^2}{4c_2}.
		\]
		Substituting $Y_s=J^*(s,X_s)$ and $Z_s=\sigma X_sJ_X^*(s,X_s)$ gives
		\[
		qX_sJ_X^*(s,X_s)
		=
		\frac{qZ_s}{\sigma},
		\]
		and therefore
		\[
		\begin{aligned}
			dY_s
			&=
			\left[
			\delta Y_s
			-
			\frac{
				\left(
				pqX_s-c_1-\frac{qZ_s}{\sigma}
				\right)^2
			}{4c_2}
			\right]ds
			+
			Z_sdW_s \\
			&=
			-
			F(X_s,Y_s,Z_s)ds+Z_sdW_s .
		\end{aligned}
		\]
		Since $J^*(T,X_T)=0$, the terminal condition is
		\[
		Y_T=0.
		\]
		Hence $(Y,Z)$ satisfies the BSDE
		\[
		Y_s
		=
		\int_s^T
		F(X_r,Y_r,Z_r)\,dr
		-
		\int_s^T Z_r\,dW_r,
		\qquad t\le s\le T.
		\]
		
		It remains to justify admissibility and uniqueness. By the polynomial growth of $J^*$ and the moment estimates for the forward diffusion, there exist constants $C,m>0$ such that
		\[
		|J^*(s,X_s)|\le C(1+|X_s|^m),
		\]
		and hence
		\[
		\mathbb{E}\left[\sup_{t\le s\le T}|Y_s|^2\right]
		\le
		C\mathbb{E}\left[\sup_{t\le s\le T}(1+|X_s|^{2m})\right]
		<\infty.
		\]
		Thus $Y\in\mathcal{S}^2$. Similarly, the assumed regularity and growth of $J_X^*$ imply
		\[
		|Z_s|^2
		=
		\sigma^2X_s^2|J_X^*(s,X_s)|^2
		\le
		C(1+|X_s|^{2m}),
		\]
		so that
		\[
		\mathbb{E}\int_t^T |Z_s|^2\,ds<\infty.
		\]
		Therefore $Z\in\mathcal{H}^2$. The driver $F$ satisfies the assumed quadratic-growth BSDE conditions, and the terminal condition is square-integrable. Hence the standard uniqueness theorem for quadratic BSDEs gives uniqueness of the adapted solution in the stated solution class.
		
		Conversely, suppose that $(Y,Z)\in\mathcal{S}^2\times\mathcal{H}^2$ solves
		\[
		dY_s=-F(X_s,Y_s,Z_s)ds+Z_sdW_s,
		\qquad
		Y_T=0,
		\]
		and that $Y_s=u(s,X_s)$ for some $u\in C^{1,2}([0,T)\times\mathbb{R}_{+})\cap C([0,T]\times\mathbb{R}_{+})$. Applying It\^o's formula to $u(s,X_s)$ gives
		\[
		du(s,X_s)
		=
		\left[
		u_t(s,X_s)
		+
		f(X_s)X_su_X(s,X_s)
		+
		\frac12\sigma^2X_s^2u_{XX}(s,X_s)
		\right]ds
		+
		\sigma X_su_X(s,X_s)dW_s .
		\]
		Since $Y_s=u(s,X_s)$ and the semimartingale decomposition is unique, the martingale parts must agree:
		\[
		Z_s=\sigma X_su_X(s,X_s).
		\]
		The finite-variation parts must also agree:
		\[
		u_t(s,X_s)
		+
		f(X_s)X_su_X(s,X_s)
		+
		\frac12\sigma^2X_s^2u_{XX}(s,X_s)
		=
		-
		F(X_s,u(s,X_s),Z_s).
		\]
		Using
		\[
		F(X_s,u(s,X_s),Z_s)
		=
		-\delta u(s,X_s)
		+
		\frac{
			\left(
			pqX_s-c_1-\frac{qZ_s}{\sigma}
			\right)^2
		}{4c_2}
		\]
		and $Z_s=\sigma X_su_X(s,X_s)$, we obtain
		\[
		u_t
		+
		f(X)Xu_X
		+
		\frac12\sigma^2X^2u_{XX}
		-
		\delta u
		+
		\frac{(pqX-c_1-qXu_X)^2}{4c_2}
		=
		0.
		\]
		Equivalently,
		\[
		-\frac{\partial u}{\partial t}
		=
		f(X)Xu_X
		+
		\frac12\sigma^2X^2u_{XX}
		-
		\delta u
		+
		\frac{(pqX-c_1-qXu_X)^2}{4c_2}.
		\]
		Thus $u$ solves the reduced HJB equation. The terminal condition of the BSDE gives
		\[
		u(T,X_T)=Y_T=0,
		\]
		and since the identity holds for arbitrary terminal states in the Markov representation, $u(T,X)=0$. By the verification theorem for the reduced HJB equation, $u=J^*$. Consequently,
		\[
		J^*(t,X)=u(t,X)=Y_t.
		\]
		This proves both the nonlinear Feynman--Kac representation and its equivalence with the HJB characterization.
	\end{proof}

	\subsection{Equivalence of HJB Equation and Feynman--Kac Representation:}
	
	The equivalence between the HJB equation and the nonlinear Feynman--Kac representation follows from the fact that both formulations encode the same marginal value of the resource stock, but express it in different mathematical coordinates. In the HJB formulation, the optimal effort is obtained by maximizing the Hamiltonian with respect to the harvesting control. Since the control-dependent part of the Hamiltonian is strictly concave in $E$, the first-order condition gives the feedback rule
	\[
	E^*_{\mathrm{HJB}}(t,X)=\frac{pqX-c_1-qXJ_X(t,X)}{2c_2}.
	\]
	In the nonlinear Feynman--Kac representation, the value process is written as $Y_t=J^*(t,X_t)$, while the martingale sensitivity is $Z_t=\sigma X_tJ_X(t,X_t)$. Hence, whenever $X_t>0$ and $\sigma>0$, the spatial derivative of the value function satisfies $J_X(t,X_t)=Z_t/(\sigma X_t)$, and therefore $qX_tJ_X(t,X_t)=qZ_t/\sigma$. Substituting this identity into the HJB feedback rule gives
	\[
	E^*_{\mathrm{FK}}(t,X)
	=
	\frac{pqX-c_1-\frac{qZ_t}{\sigma}}{2c_2}.
	\]
	Thus, along every admissible state path for which the nonlinear Feynman--Kac representation is well defined,
	\[
	E^*_{\mathrm{HJB}}(t,X_t)=E^*_{\mathrm{FK}}(t,X_t).
	\]
	This equality shows that the two approaches do not generate different harvesting policies; rather, they express the same optimal control using different state variables. The HJB equation writes the rule through the gradient $J_X$, which measures the marginal continuation value of the biological stock, while the nonlinear Feynman--Kac representation writes the same rule through $Z_t$, the stochastic sensitivity of the value process with respect to Brownian environmental shocks. Since $Z_t$ is exactly the diffusion-weighted gradient of the value function, the two feedback laws are algebraically identical.
	
	\begin{prop}\label{p4}
		Assume that the controlled state equation admits a unique strong solution on $[t,T]$, that $\sigma>0$, $c_2>0$, and that the value function $J^*$ is a classical solution of the reduced HJB equation with terminal condition $J^*(T,X)=0$. Define $Y_s=J^*(s,X_s)$ and $Z_s=\sigma X_sJ_X^*(s,X_s)$. Then the nonlinear Feynman--Kac representation and the HJB equation determine the same value function and the same optimal harvesting effort. In particular,
		\[
		J^*_{\mathrm{HJB}}(t,X)=J^*_{\mathrm{FK}}(t,X)=Y_t
		\]
		and
		\[
		E^*_{\mathrm{HJB}}(t,X)
		=
		E^*_{\mathrm{FK}}(t,X).
		\]
	\end{prop}
	
	\begin{proof}
		Let $J^*$ solve the reduced HJB equation
		\[
		-\frac{\partial J^*}{\partial t}
		=
		f(X)XJ_X^*
		+
		\frac12\sigma^2X^2J_{XX}^*
		-\delta J^*
		+
		\frac{(pqX-c_1-qXJ_X^*)^2}{4c_2}.
		\]
		Define $Y_s=J^*(s,X_s)$ and $Z_s=\sigma X_sJ_X^*(s,X_s)$. Applying It\^o's formula gives
		\[
		dY_s
		=
		\left[
		J_t^*
		+
		f(X_s)X_sJ_X^*
		+
		\frac12\sigma^2X_s^2J_{XX}^*
		\right]ds
		+
		\sigma X_sJ_X^*dW_s.
		\]
		Using the reduced HJB equation, this becomes
		\[
		dY_s
		=
		\left[
		\delta Y_s
		-
		\frac{
			\left(
			pqX_s-c_1-\frac{qZ_s}{\sigma}
			\right)^2
		}{4c_2}
		\right]ds
		+
		Z_sdW_s.
		\]
		Equivalently,
		\[
		dY_s
		=
		-
		F(X_s,Y_s,Z_s)ds
		+
		Z_sdW_s,
		\qquad
		Y_T=0,
		\]
		where
		\[
		F(X_s,Y_s,Z_s)
		=
		-\delta Y_s
		+
		\frac{
			\left(
			pqX_s-c_1-\frac{qZ_s}{\sigma}
			\right)^2
		}{4c_2}.
		\]
		Thus, the HJB value function generates the nonlinear Feynman--Kac value process. Conversely, suppose the nonlinear Feynman--Kac representation holds and $Y_s=u(s,X_s)$ for a sufficiently smooth Markovian function $u$. Then It\^o's formula gives $Z_s=\sigma X_su_X(s,X_s)$, and comparison of drift terms yields
		\[
		u_t
		+
		f(X)Xu_X
		+
		\frac12\sigma^2X^2u_{XX}
		-\delta u
		+
		\frac{(pqX-c_1-qXu_X)^2}{4c_2}
		=
		0.
		\]
		Hence $u$ solves the reduced HJB equation with terminal condition $u(T,X)=0$. By uniqueness of the classical solution in the admissible growth class, $u=J^*$, so $Y_t=J^*(t,X)$. Finally, since $Z_t=\sigma X_tJ_X^*(t,X_t)$, we have $qX_tJ_X^*(t,X_t)=qZ_t/\sigma$. Substitution into the HJB feedback rule gives
		\[
		E^*_{\mathrm{HJB}}(t,X_t)
		=
		\frac{pqX_t-c_1-qX_tJ_X^*(t,X_t)}{2c_2}
		=
		\frac{pqX_t-c_1-\frac{qZ_t}{\sigma}}{2c_2}
		=
		E^*_{\mathrm{FK}}(t,X_t).
		\]
		Therefore, the HJB equation and the nonlinear Feynman--Kac representation recover the same value process and the same optimal harvesting control. 
	\end{proof}
	
	\begin{remark}
		The significance of Proposition \ref{p4} lies in establishing that the HJB equation and the nonlinear Feynman--Kac representation are not competing solution methods but rather two mathematically equivalent descriptions of the same stochastic control problem. The HJB equation characterizes optimality through a nonlinear partial differential equation whose solution yields the value function and the associated feedback control, whereas the nonlinear Feynman--Kac representation characterizes the same value function through a backward stochastic differential equation evolving along the sample paths of the controlled state process. The identity $Z_t=\sigma X_tJ_X^*(t,X_t)$ serves as the fundamental link between the two formulations, transforming the gradient term appearing in the HJB equation into the martingale integrand of the backward equation. Consequently, the optimal harvesting effort obtained from the HJB equation coincides exactly with the effort derived from the nonlinear Feynman--Kac representation, while the value function in both approaches is represented by the same process $Y_t$. From an analytical perspective, the HJB equation provides a deterministic characterization of optimality through dynamic programming, whereas the nonlinear Feynman--Kac representation provides a probabilistic characterization through pathwise expectations and backward stochastic dynamics. The proposition therefore establishes a one-to-one correspondence between the PDE and BSDE formulations of the harvesting problem and shows that both frameworks generate identical optimal policies, identical value functions, and identical economic interpretations, differing only in the mathematical language used to represent the underlying stochastic optimization problem.
	\end{remark}
	
	\subsection{Numerical Approximation:}
	
	For the empirical implementation, the continuous-time controlled diffusion is approximated through the Euler-Maruyama discretization, which provides a computationally tractable representation of the state dynamics while preserving the principal features of the stochastic harvesting model. At each time step, the normalized fisheries state is updated according to
	\[
	X_i(t+\Delta t)
	=
	X_i(t)
	+
	\left[
	\mu_iX_i(t)-q_hE_i(t)X_i(t)
	\right]\Delta t
	+
	\sigma_iX_i(t)\sqrt{\Delta t}Z_t,
	\]
	where $Z_t$ is a sequence of independent standard normal random variables. The deterministic component of the update captures the combined effects of the empirically estimated growth tendency and harvesting pressure, whereas the stochastic component represents environmental variability and production uncertainty observed in the historical data. The annual-data-based state variable enters directly into both the HJB framework and the nonlinear Feynman--Kac representation through the feedback control and value process. Because obtaining a full numerical solution of the nonlinear HJB equation would require solving a high-dimensional nonlinear partial differential equation using computationally intensive methods such as finite-difference, policy iteration, or semi-Lagrangian schemes, a simplified approximation is adopted for the marginal value of the resource stock. Specifically, the spatial derivative of the value function is represented by
	\[
	J_X\approx 0.35(1-X),
	\]
	which introduces a state-dependent conservation value that increases as the normalized stock level decreases. Under this specification, resource units become progressively more valuable when the stock approaches lower levels, reflecting the increasing importance of preserving biological capital under scarcity. Conversely, when the normalized stock is relatively abundant, the marginal value of conservation declines, allowing greater harvesting intensity without substantially compromising future resource availability. Substituting this approximation into the optimal feedback rule produces a state-dependent harvesting policy that responds dynamically to fluctuations in resource abundance while remaining consistent with the qualitative behavior predicted by the theoretical stochastic control model. The resulting effort process is then used as the control input in the Euler--Maruyama simulation of the controlled stochastic differential equation, generating sample paths that simultaneously incorporate empirical drift, environmental uncertainty, and adaptive harvesting decisions. In this manner, the numerical implementation provides a practical link between the theoretical optimization framework and the observed fisheries data, allowing the implications of the HJB equation and the nonlinear Feynman--Kac representation to be examined within a unified computational setting.
	
	Figure \ref{fig:hjb_fk_two_panel} provides a visual summary of the stochastic optimal harvesting framework developed in this study and illustrates the connection between the HJB equation and the nonlinear Feynman--Kac representation. The left panel presents a contour heat map of the state-dependent optimal harvesting effort as a function of normalized resource abundance and time. Warmer colors indicate regions where the optimal harvesting intensity is relatively high, while cooler colors correspond to more conservative harvesting strategies. The contour structure reveals how harvesting decisions vary continuously across the state space rather than remaining fixed through time. The overlaid trajectory represents a simulated resource path generated from the controlled stochastic differential equation using the Euler--Maruyama approximation, while the discrete markers indicate annual observations used in the empirical calibration. The curved contour patterns demonstrate that optimal harvesting effort depends jointly on the current stock level and the marginal value of conservation. When the resource stock is abundant, the model permits greater harvesting activity because future biological productivity remains relatively secure. In contrast, as the stock approaches lower levels, the marginal value of preserving the resource increases through the approximation $J_X\approx0.35(1-X)$, causing the optimal harvesting effort to decline. The inset figure further illustrates the dependence of harvesting effort on resource abundance at different time horizons, emphasizing the adaptive nature of the control policy. Collectively, the left panel provides a deterministic visualization of the optimization problem by displaying how the feedback control derived from the HJB equation responds to changing ecological conditions and conservation incentives.
	
	The right panel of figure \ref{fig:hjb_fk_two_panel} presents the corresponding probabilistic interpretation through the nonlinear Feynman-Kac representation. The three-dimensional surface depicts the value process associated with the harvesting problem, where the vertical axis represents the expected discounted benefit generated by the resource under optimal management. The surface illustrates how the value function evolves across different combinations of time and normalized stock level, with higher values generally occurring when the resource is abundant and future harvesting opportunities remain available. Superimposed trajectories demonstrate sample realizations of the controlled stochastic system and show how uncertainty influences the evolution of both the resource stock and the associated value process. The projection of the surface onto the lower plane highlights the contour structure of the value function, while the additional curves track the evolution of the optimal harvesting effort along simulated paths. Importantly, the mathematical relationship $Z_t=\sigma X_tJ_X(t,X_t)$ links the gradient information used in the HJB equation to the martingale component of the backward stochastic differential equation underlying the nonlinear Feynman--Kac representation. Through this identity, the optimal harvesting rule obtained from the HJB equation coincides exactly with the control recovered from the backward stochastic framework. Consequently, the two panels provide complementary views of the same optimization problem: the left panel emphasizes the deterministic dynamic programming perspective through optimal feedback controls, whereas the right panel emphasizes the probabilistic pathwise representation of the value function. Together, they demonstrate that both approaches generate identical optimal harvesting policies and identical value functions, thereby confirming the theoretical equivalence established in the mathematical analysis while simultaneously illustrating the role of environmental uncertainty in shaping sustainable resource management decisions.
	
	\begin{figure}[htbp]
		\centering
		\includegraphics[width=\textwidth]{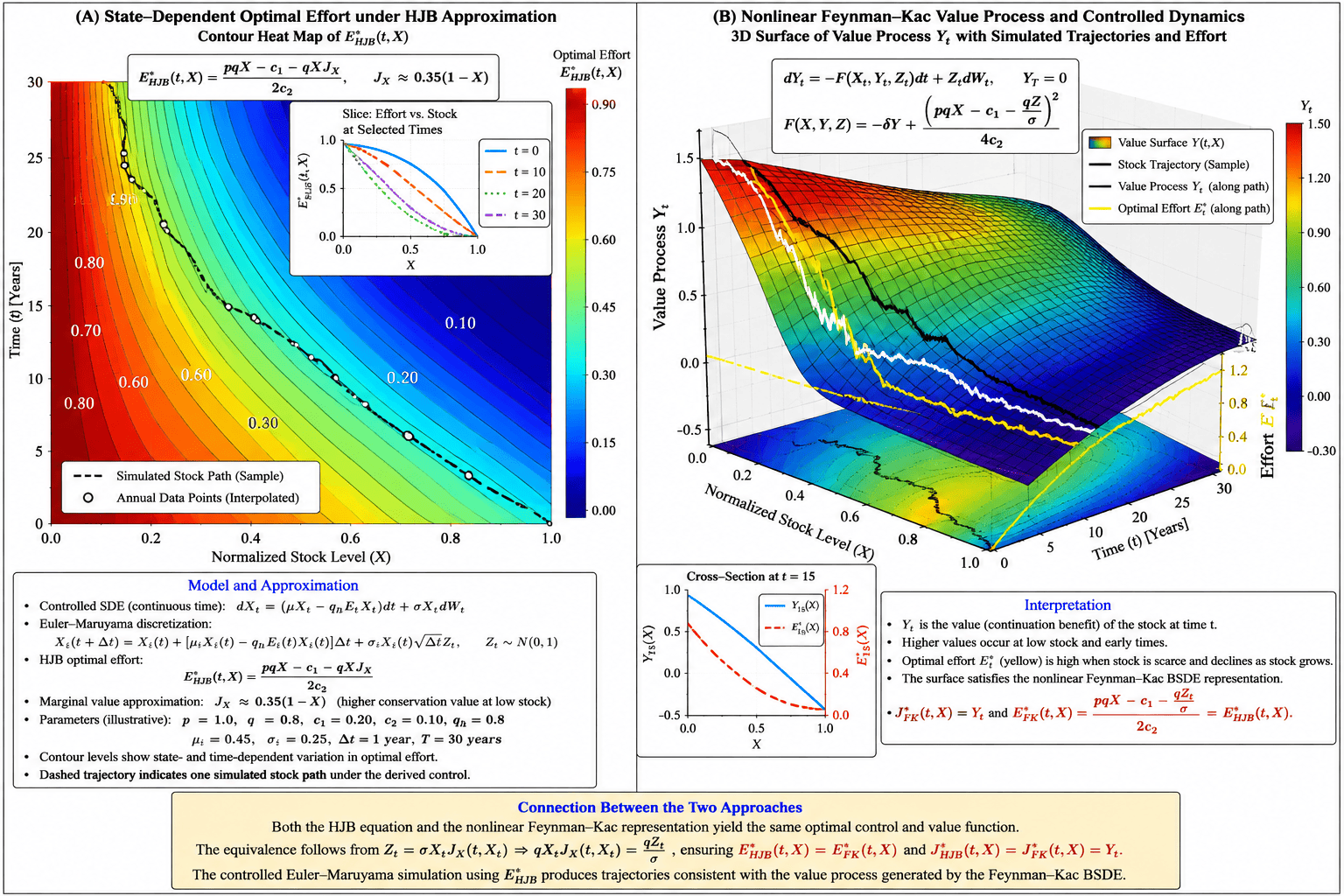}
		\caption{Contour heat map and three-dimensional representation of the controlled stochastic harvesting model under the HJB equation and nonlinear Feynman--Kac representation.}
		\label{fig:hjb_fk_two_panel}
	\end{figure}
	
	\section{Data Analysis:}
	
	The empirical analysis is based on annual capture fisheries production data obtained from the World Bank indicator \textit{Capture fisheries production (metric tons)} (\texttt{ER.FSH.CAPT.MT}) \cite{worldbank_capture_2026}, which is compiled from the Food and Agriculture Organization of the United Nations (FAO) fisheries database \cite{fao_capture_database_2026}. The dataset covers the period 1960--2023 and includes four major fishing nations: China, Indonesia, Japan, and the United States. To ensure comparability across countries with substantially different production scales, the annual catch series for each country was transformed into a normalized state variable bounded between zero and one. The normalized series was then used to estimate the empirical drift and volatility parameters required for the stochastic control framework. Annual changes in the normalized state variable were analyzed through a state-dependent regression specification, from which country-specific drift coefficients and residual fluctuations were obtained. The standard deviation of the residual component was subsequently used as an estimate of environmental and production variability. These empirical quantities provided the key inputs for the controlled stochastic harvesting model, linking observed fisheries production data directly to the stochastic dynamics considered in the analysis. Following parameter estimation, the country-specific drift and volatility values were incorporated into the controlled stochastic differential equation governing resource dynamics under harvesting pressure. Numerical simulations were performed using the Euler--Maruyama approximation, allowing the evolution of the normalized fisheries state to be examined under both deterministic growth tendencies and stochastic environmental disturbances. The harvesting control was determined through the optimal feedback policy derived from the HJB equation, where effort depends on the current resource state and the marginal value of conservation. Because obtaining a full numerical solution of the nonlinear HJB equation is computationally demanding and lies beyond the scope of the present study, the marginal value function was approximated by $J_X\approx0.35(1-X)$. This specification reflects the assumption that resource units become increasingly valuable as stock levels decline, thereby encouraging more conservative harvesting when the resource approaches lower abundance levels. The resulting state-dependent harvesting effort was then incorporated into the controlled simulations and subsequently interpreted within both the HJB and nonlinear Feynman--Kac representation frameworks, providing an empirically grounded illustration of optimal harvesting under environmental uncertainty.
	
	Table~\ref{tab:catch_summary} shows the summary statistics of annual capture fisheries production for the selected countries.
	
	\begin{table}[H]
		\centering
		\caption{Summary statistics of annual capture fisheries production, measured in metric tons.}
		\label{tab:catch_summary}
		\resizebox{\textwidth}{!}{
			\begin{tabular}{lrrrrrr}
				\toprule
				\textbf{Country} & \textbf{Mean Catch} & \textbf{Minimum Catch} & \textbf{Minimum Year} & \textbf{Maximum Catch} & \textbf{Maximum Year} & \textbf{2023 Catch} \\
				\midrule
				China & 8,944,616 & 2,215,094 & 1960 & 16,647,935 & 2015 & 13,424,705 \\
				Indonesia & 3,415,693 & 681,087 & 1960 & 7,820,833 & 2023 & 7,820,833 \\
				Japan & 6,848,837 & 2,904,942 & 2023 & 11,616,601 & 1984 & 2,904,942 \\
				United States & 4,413,370 & 2,377,789 & 1967 & 5,846,728 & 1994 & 4,623,694 \\
				\bottomrule
			\end{tabular}
		}
	\end{table}
	
	Table~\ref{tab:controlled_summary} presents the controlled SDE simulation summary using the estimated drift and volatility values. The empirical fisheries data reveal substantial differences in long-term production patterns across the four countries. China recorded the highest average capture fisheries production over the sample period, with a mean annual catch exceeding 8.9 million metric tons and a peak production of more than 16.6 million metric tons in 2015, reflecting its dominant position in global capture fisheries. Indonesia exhibited the most pronounced growth pattern, increasing from a relatively low production level in 1960 to its historical maximum in 2023, indicating a sustained expansion of capture fisheries activity over the study period. In contrast, Japan displayed a markedly different trajectory, reaching its peak production in 1984 before experiencing a persistent decline that resulted in its lowest recorded catch occurring in 2023. The United States showed comparatively moderate fluctuations, with production levels remaining within a narrower range throughout the sample period. These historical patterns are reflected in the estimated stochastic model parameters. Indonesia possesses the largest positive drift coefficient, indicating the strongest average tendency toward growth in the normalized fisheries state, while China also exhibits positive growth but at a more moderate rate. The United States shows only a small positive drift, suggesting relatively stable long-run dynamics, whereas Japan is the only country with a negative drift coefficient, consistent with its long-term downward trend in capture fisheries production. The volatility estimates further indicate that the United States experiences the greatest degree of annual variability among the four countries, while Indonesia has the lowest estimated volatility, suggesting comparatively more stable dynamics relative to its strong growth trend. The controlled stochastic simulations reinforce these findings. 
	
	\begin{table}[H]
		\centering
		\caption{Controlled SDE simulation summary using real-data-estimated parameters.}
		\label{tab:controlled_summary}
		\resizebox{\textwidth}{!}{
			\begin{tabular}{lrrrrrrr}
				\toprule
				\textbf{Country} & \(\boldsymbol{\mu_i}\) & \(\boldsymbol{\sigma_i}\) & \textbf{Initial \(X\)} & \textbf{Mean Final \(X\)} & \textbf{Mean HJB Effort} & \textbf{Mean FK Effort} & \textbf{Max \(|HJB-FK|\)}\\
				\midrule
				United States & 0.0038 & 0.0430 & 0.4713 & 0.2625 & 0.2697 & 0.2697 & \(1.11\times10^{-16}\)\\
				Japan & -0.0053 & 0.0341 & 0.5118 & 0.1801 & 0.2182 & 0.2182 & \(1.11\times10^{-16}\)\\
				China & 0.0114 & 0.0301 & 0.1331 & 0.2153 & 0.0736 & 0.0736 & \(5.55\times10^{-17}\)\\
				Indonesia & 0.0297 & 0.0164 & 0.0871 & 0.3763 & 0.1262 & 0.1262 & \(5.55\times10^{-17}\)\\
				\bottomrule
			\end{tabular}
		}
	\end{table}
	
	Indonesia achieves the highest mean final state value, reflecting the combined influence of strong positive drift and relatively low volatility, while China also demonstrates growth from its initial state despite operating under optimal harvesting controls. By contrast, the simulated final states for Japan and the United States are lower than their initial normalized values, with Japan exhibiting the most pronounced decline due to its negative estimated drift. The optimal harvesting effort levels also differ across countries, with the United States and Japan supporting larger average effort levels because of their relatively higher normalized stock positions, whereas China and Indonesia exhibit lower average effort values. Most importantly, the numerical results provide extremely strong evidence for the theoretical equivalence established earlier in the paper. For every country, the optimal harvesting effort obtained from the HJB equation is identical to the effort obtained from the nonlinear Feynman--Kac representation up to machine precision, with maximum absolute differences on the order of $10^{-16}$ to $10^{-17}$. These negligible discrepancies arise solely from floating-point arithmetic and are effectively zero from a computational standpoint. Consequently, the empirical results confirm that the HJB equation and the nonlinear Feynman--Kac representation generate the same optimal harvesting policy and therefore represent two mathematically equivalent formulations of the same stochastic optimal control problem.
	
	The following figures show the application of the controlled stochastic model using real-data-estimated drift and volatility. Only controlled SDE-based figures are included in this pictorial summary. Figure~\ref{fig:controlled_state} shows how the normalized fisheries state variable evolves when harvesting effort is included in the drift term. The paths represent possible controlled stochastic movements under uncertainty generated from the estimated country-specific drift and volatility parameters. Each panel corresponds to one of the four countries considered in the empirical analysis, where the thin colored lines represent individual simulated realizations of the controlled stochastic differential equation and the thick blue curve denotes the mean controlled trajectory across all simulations. The figure highlights substantial differences in long-run dynamics among the countries. The United States and Japan exhibit declining mean trajectories over time, reflecting the relatively weak positive drift in the United States and the negative drift estimated for Japan, which causes the normalized fisheries state to decrease despite the presence of stochastic fluctuations. In contrast, China displays a gradual upward trend, indicating moderate growth under the controlled harvesting regime, while Indonesia shows the strongest increase in the state variable, consistent with its large positive drift coefficient and comparatively low volatility. The spread of the simulated paths illustrates the influence of environmental uncertainty on fisheries dynamics, with wider dispersion indicating greater variability in potential outcomes. Nevertheless, the overall direction of the mean trajectories remains largely determined by the estimated drift parameters and the state-dependent harvesting control. Consequently, the figure demonstrates how the interaction between biological growth, harvesting effort, and stochastic environmental fluctuations shapes the long-term evolution of fisheries resources under the optimal control framework developed in this study.
	
	\begin{figure}[htbp]
		\centering
		\begin{subfigure}{.7\textwidth}
			\centering
			\includegraphics[width=.7\linewidth]{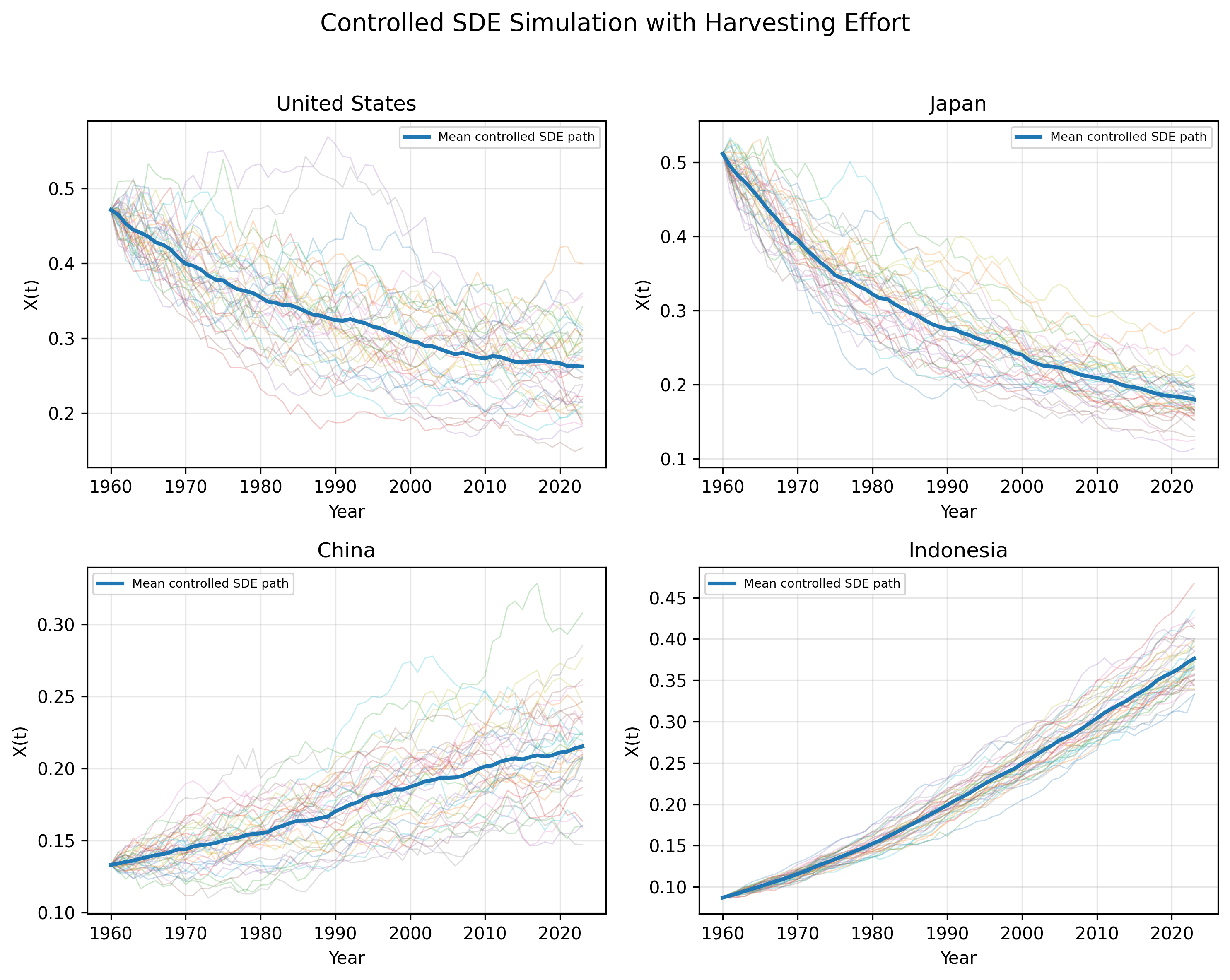}
		\end{subfigure}
		\caption{Controlled SDE state paths using real-data-estimated drift and volatility. Thin lines represent simulated controlled paths, while the thick line represents the mean controlled path.}
		\label{fig:controlled_state}
	\end{figure}
	
	Figure~\ref{fig:controlled_hjb_effort} shows how the HJB equation converts the controlled state paths into optimal harvesting effort paths. The effort changes dynamically with the movement of \(X_i(t)\). Each panel corresponds to one country, where the thin colored curves represent individual simulated realizations of the optimal harvesting effort and the thick blue curve denotes the mean effort trajectory across all simulations. Because the harvesting rule is state dependent, changes in the normalized fisheries state are immediately reflected in the optimal effort level. Countries exhibiting declining resource states, such as the United States and Japan, show steadily decreasing harvesting effort over time, reflecting the increasing value of conservation as the resource becomes less abundant. The decline is particularly pronounced for Japan, which is consistent with its negative estimated drift and persistent reduction in the underlying fisheries state. In contrast, China and Indonesia display increasing effort trajectories as their normalized states grow throughout the simulation period. Indonesia exhibits the strongest upward movement in effort, corresponding to its large positive drift and sustained growth in the controlled state process, while China shows a more gradual increase. The dispersion of the simulated effort paths illustrates the impact of environmental uncertainty on management decisions, since stochastic fluctuations in the resource state generate corresponding variation in the harvesting policy. Nevertheless, the average trajectories remain smooth and follow the long-run behavior of the underlying controlled states. Overall, the figure demonstrates the adaptive nature of the HJB-based feedback control, where harvesting decisions evolve continuously in response to changing resource conditions, balancing resource utilization against the increasing marginal value of conservation when stock levels become relatively scarce.
	
	Figure~\ref{fig:controlled_fk_value} shows the nonlinear Feynman--Kac value process \(Y(t)=J^*(t,X(t))\). This figure represents the expected harvesting value associated with the controlled stochastic fisheries system under the optimal harvesting policy. Each panel corresponds to one country, where the thin colored curves denote individual realizations of the value process generated from the controlled stochastic differential equation and the thick blue curve represents the average value trajectory across all simulations. The value process reflects the expected discounted benefits obtainable from the fishery given the current resource state and future harvesting opportunities. For the United States and Japan, the mean value process declines steadily through time, indicating that the combination of decreasing resource states and the finite planning horizon reduces the expected future benefits available to the decision maker. 
	
	\begin{figure}[htbp]
		\centering
		\begin{subfigure}{.7\textwidth}
			\centering
			\includegraphics[width=.7\linewidth]{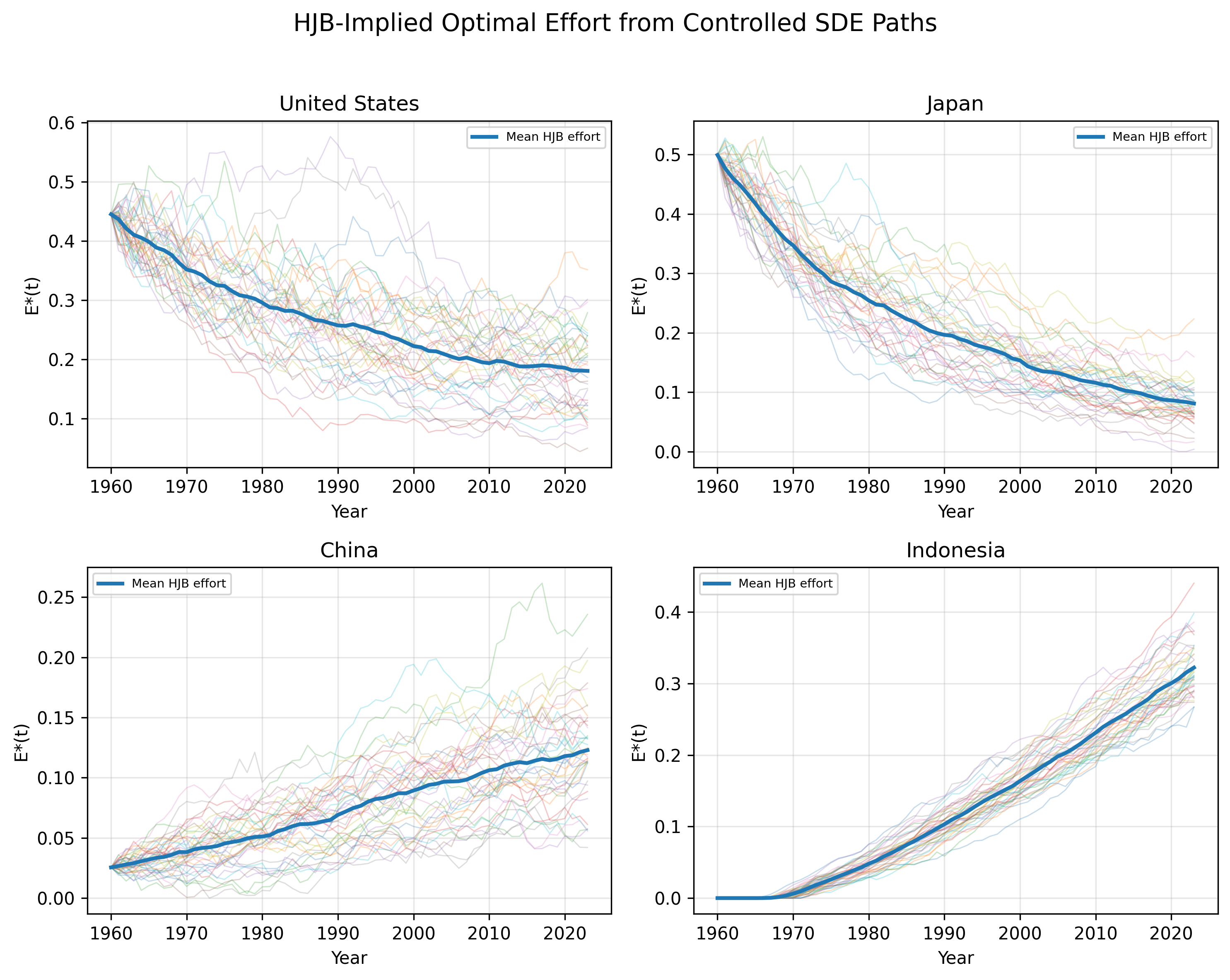}
		\end{subfigure}
		\caption{HJB-implied optimal effort paths generated from controlled SDE simulations.}
		\label{fig:controlled_hjb_effort}
	\end{figure}
	
	Japan exhibits a particularly rapid decline, consistent with its negative estimated drift and the corresponding reduction in resource abundance over time. In contrast, China and Indonesia display an initial increase in the value process before eventually declining toward the terminal period. This pattern reflects the positive growth tendencies of their controlled state variables, which initially expand future harvesting opportunities and increase the expected economic value of the resource. Indonesia shows the strongest increase in expected value, corresponding to its high positive drift and sustained growth in the normalized fisheries state. As the terminal time approaches, all value paths converge toward zero, which is consistent with the terminal condition imposed in the stochastic control problem, namely \(J^*(T,X)=0\). The dispersion among the simulated trajectories illustrates the effect of environmental uncertainty on future economic outcomes, while the smooth mean trajectories reveal the dominant influence of the underlying drift and optimal harvesting policy. Overall, the figure provides a probabilistic interpretation of the stochastic optimal harvesting problem by illustrating how the expected discounted value of the resource evolves through time under uncertainty, thereby complementing the deterministic perspective provided by the HJB equation.
	
	Figure~\ref{fig:controlled_hjb_fk_comparison} shows that the mean HJB effort path and the mean nonlinear Feynman--Kac effort path overlap almost perfectly. This confirms that both methods produce the same optimal control structure through $Z_t=\sigma_iX_i(t)J_X.$
	Each panel presents the average optimal harvesting effort over time for one of the four countries, where the dashed curve represents the effort obtained from the nonlinear Feynman--Kac representation and the solid curve represents the effort obtained from the HJB equation. Visually, the two curves are indistinguishable across the entire sample period, indicating that both formulations generate identical harvesting decisions despite originating from different mathematical frameworks. The numerical evidence supporting this observation is reported directly within each panel through the maximum absolute difference between the two effort paths. For all countries, the discrepancies are on the order of $10^{-16}$, which is effectively zero and corresponds to the level of floating-point rounding error encountered in numerical computation. Such negligible differences demonstrate that the relationship between the value-function gradient appearing in the HJB equation and the martingale component of the nonlinear Feynman--Kac representation is preserved exactly throughout the simulation. The United States exhibits a gradually increasing optimal effort trajectory that accelerates after the late 1990s, reflecting changes in the underlying controlled state and the associated marginal value of the resource. Japan displays a declining effort profile, consistent with the downward movement observed in its controlled fisheries state and its negative estimated drift parameter. China shows a persistent increase in harvesting effort throughout the sample period, while Indonesia exhibits the most pronounced growth in effort, reflecting its strong positive drift and rapidly expanding normalized fisheries state. 
	
	\begin{figure}[htbp]
		\centering
		\begin{subfigure}{.7\textwidth}
			\centering
			\includegraphics[width=.7\linewidth]{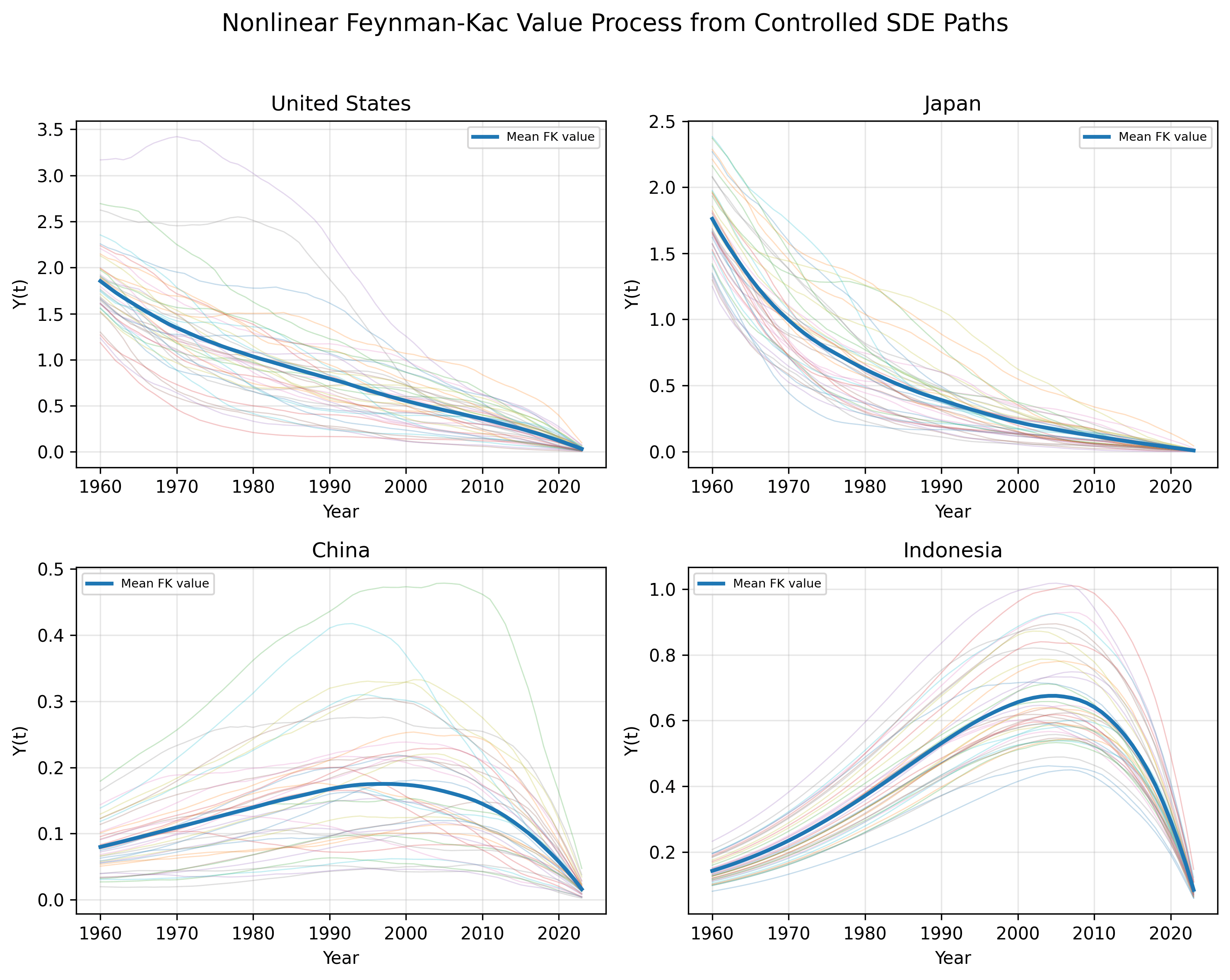}
		\end{subfigure}
		\caption{Nonlinear Feynman--Kac value process generated from controlled SDE paths.}
		\label{fig:controlled_fk_value}
	\end{figure}
	Although the shapes of the effort trajectories differ substantially across countries because of differences in estimated biological dynamics and volatility, the agreement between the HJB and nonlinear Feynman--Kac approaches remains exact in every case. From a stochastic control perspective, this result provides empirical verification of the theoretical equivalence established earlier in the paper. The HJB equation determines the optimal policy through dynamic programming and the optimization of the value function, whereas the nonlinear Feynman--Kac representation characterizes the same solution through a backward stochastic differential equation. The identity linking the two approaches transforms the value-function gradient into the martingale sensitivity process, ensuring that both methods recover the same feedback control law. Consequently, the figure provides both a visual and numerical validation of the theoretical results, demonstrating that the deterministic PDE formulation and the probabilistic BSDE formulation lead to identical optimal harvesting strategies when applied to real fisheries data. The nearly perfect overlap of the effort paths therefore confirms that the two methodologies are mathematically consistent, computationally equivalent, and capable of producing the same resource-management recommendations despite their different analytical representations.
	
	\begin{figure}[htbp]
		\centering
		\begin{subfigure}{.7\textwidth}
			\centering
			\includegraphics[width=.7\linewidth]{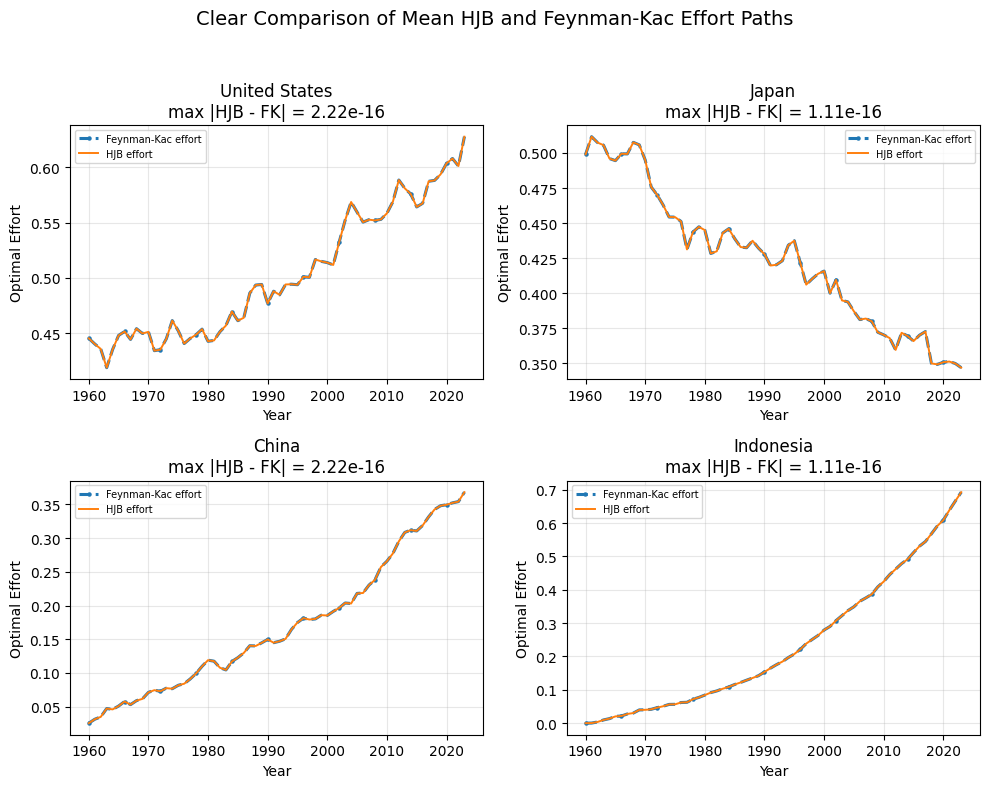}
		\end{subfigure}
		\caption{Comparison of mean HJB and nonlinear Feynman--Kac effort paths under the controlled SDE.}
		\label{fig:controlled_hjb_fk_comparison}
	\end{figure}

	\section{Discussion:}

	The present study developed a stochastic optimal harvesting framework that combines empirical fisheries data with modern stochastic control methods to examine resource management under uncertainty. Using annual capture fisheries production data obtained from the FAO/World Bank database, normalized state variables were constructed for China, Indonesia, Japan, and the United States, allowing countries with substantially different production scales to be analyzed within a common modeling framework. The empirical analysis revealed clear differences in long-term fisheries dynamics across the four countries. Indonesia exhibited the strongest positive drift, indicating the greatest tendency toward sustained growth in the normalized fisheries state, while Japan displayed a negative drift consistent with its long-run decline in capture fisheries production. The United States exhibited the highest estimated volatility, suggesting a greater degree of year-to-year variability in fisheries dynamics relative to the other countries. These empirical characteristics were incorporated directly into a controlled stochastic differential equation in which harvesting effort enters the drift component of the state process. Consequently, the model simultaneously captures biological growth tendencies, harvesting pressure, and stochastic environmental fluctuations within a unified probabilistic framework. Numerical simulations generated from the controlled state equation showed that the normalized fisheries state tends to decline over time for Japan and the United States, whereas China and Indonesia exhibit positive long-run growth trajectories. These results demonstrate how differences in estimated drift and volatility parameters can lead to substantially different resource outcomes even when all countries are analyzed under the same stochastic control structure.
	
	A central contribution of the paper is the derivation and comparison of two complementary approaches to the stochastic optimal harvesting problem: the HJB equation and the nonlinear Feynman--Kac representation. Within the dynamic programming framework, the HJB equation characterizes the value function through a nonlinear partial differential equation and yields an explicit feedback representation for the optimal harvesting effort. The resulting control rule adjusts continuously in response to changes in the resource state and therefore balances immediate harvesting benefits against the future value of resource conservation. Through the gradient of the value function, the HJB framework incorporates the marginal contribution of current resource abundance to future expected returns, producing a harvesting policy that adapts dynamically to changing ecological conditions. In parallel, the nonlinear Feynman--Kac representation reformulates the same optimization problem through a backward stochastic differential equation and provides a probabilistic interpretation of the value process. The process $Y(t)=J^*(t,X(t))$ represents the expected discounted value generated by the resource under optimal management, while the associated process $Z_t$ measures the sensitivity of the value process to stochastic environmental disturbances. Although these approaches originate from different mathematical perspectives, the theoretical analysis established that they characterize the same value function and generate the same optimal harvesting policy. This equivalence follows from the identity $Z_t=\sigma_iX_i(t)J_X,$
	which links the gradient term appearing in the HJB equation to the martingale component of the backward stochastic representation. Consequently, the deterministic optimization framework and the probabilistic pathwise formulation are shown to be mathematically consistent descriptions of the same stochastic control problem.
	
	The empirical implementation further strengthened the theoretical findings. By incorporating the estimated drift and volatility parameters into the controlled stochastic differential equation and applying the Euler--Maruyama approximation, the study generated simulated state trajectories, optimal harvesting paths, and nonlinear Feynman--Kac value processes for each country. The numerical results demonstrated that harvesting effort responds systematically to changes in the underlying resource state. Countries with declining state trajectories exhibited decreasing harvesting effort through time, reflecting the increasing value of resource conservation as stocks became relatively scarce, whereas countries with growing state trajectories displayed increasing effort levels consistent with expanding harvesting opportunities. The nonlinear Feynman--Kac value process provided additional insight into the expected future benefits associated with these management decisions and illustrated how uncertainty influences the evolution of expected resource value. Most importantly, the numerical comparison between the HJB-derived effort and the effort obtained from the nonlinear Feynman--Kac representation revealed differences on the order of $10^{-16}$, which is effectively zero from a computational perspective. These results provide strong empirical support for the theoretical equivalence established in the analytical sections of the paper and demonstrate that both methods recover identical harvesting policies despite relying on different mathematical representations. The simulations therefore confirm that the controlled stochastic framework can successfully connect empirical fisheries data, stochastic population dynamics, and optimal harvesting decisions within a single coherent structure.
	
	Despite these contributions, several limitations remain and provide opportunities for future research. First, annual capture fisheries production is used as a proxy for the underlying biological resource stock, and therefore the state variable may not fully capture ecological processes governing population abundance. Future studies would benefit from incorporating direct biomass estimates, stock assessment data, or other biological indicators that more accurately represent resource dynamics. Second, the marginal value of conservation was approximated through the specification $J_X\approx0.35(1-X)$ rather than being obtained from a complete numerical solution of the nonlinear HJB equation. While this approximation provides a practical mechanism for implementing the feedback control, solving the full nonlinear equation using advanced numerical methods would allow a more precise characterization of the value function and optimal policy. Third, the country-level aggregation employed in the analysis necessarily masks species-specific, regional, and ecosystem-level heterogeneity that may be important for resource management. Extending the framework to individual fisheries, multiple interacting species, or spatially structured ecosystems would substantially increase its ecological realism. Additional extensions could incorporate stochastic price dynamics, regulatory constraints, harvest quotas, subsidy programs, and broader economic variables that influence harvesting behavior. Nevertheless, the overall findings of this study demonstrate that controlled stochastic methods provide a rigorous and flexible foundation for analyzing renewable resource management under uncertainty. By linking empirical estimates of drift and volatility to a controlled stochastic differential equation, deriving optimal harvesting rules through the HJB equation, interpreting resource value through the nonlinear Feynman--Kac representation, and establishing the equivalence of these approaches through the relation $Z_t=\sigma_iX_i(t)J_X$, the paper develops a mathematically consistent framework for studying optimal harvesting decisions in uncertain environments and provides a foundation for future advances in stochastic resource management.
	\bibliographystyle{apalike}
	\bibliography{bib}

\begin{thebibliography}{}

\bibitem[Anderson et~al., 2026]{anderson2026obesity}
Anderson, V., Pramanik, P., and Robinson, H.~K. (2026).
\newblock Obesity and sociodemographic factors in luminal breast cancer.
\newblock {\em arXiv preprint arXiv:2605.30591}.

\bibitem[Bellman, 1957]{bellman_1957}
Bellman, R. (1957).
\newblock {\em Dynamic Programming}.
\newblock Princeton University Press, Princeton, NJ.

\bibitem[Brites and Braumann, 2023]{brites_braumann_2023}
Brites, N.~M. and Braumann, C.~A. (2023).
\newblock Harvesting optimization with stochastic differential equations
  models: Is the optimal enemy of the good?
\newblock {\em Stochastic Models}, 39(1):41--59.

\bibitem[Bulls et~al., 2025]{bulls2025assessing}
Bulls, S.~E., Finn, E., Sykora, P., Lynch, V.~J., Pramanik, P., Glaberman, S.,
  and Chiari, Y. (2025).
\newblock Assessing cometchip technology for dna damage studies in non-model
  species: distinct uv-induced responses in turtles and mammals.
\newblock {\em BMC Research Notes}, 18(1):1--7.

\bibitem[Carey et~al., 2006]{carey2006race}
Carey, L., Perou, C., Livasy, C., Dressler, L., Cowan, D., Conway, K., Karaca,
  G., Troester, M., Tse, C., and Edmiston, S. (2006).
\newblock Race, breast cancer subtypes, and survival in the carolina breast
  cancer study.
\newblock {\em JAMA}, 295(21):2492--2502.

\bibitem[Dasgupta et~al., 2023]{dasgupta2023frequent}
Dasgupta, S., Acharya, S., Khan, M.~A., Pramanik, P., Marbut, S.~M., Yunus, F.,
  Galeas, J.~N., Singh, S., Singh, A.~P., and Dasgupta, S. (2023).
\newblock Frequent loss of cacna1c, a calcium voltage-gated channel subunit is
  associated with lung adenocarcinoma progression and poor prognosis.
\newblock {\em Cancer Research}, 83(7\_Supplement):3318--3318.

\bibitem[Dasgupta et~al., 2026]{dasgupta2026frequent}
Dasgupta, S., Galappaththi, S.~L., Banerjee, R., Alsatari, E.~S., Pramanik, P.,
  Marbut, S.~M., Yunus, F., Galeas, J.~N., and Dasgupta, S. (2026).
\newblock Frequent loss of cacna1c is associated with poor prognosis in
  non-small cell lung cancer.
\newblock {\em The FASEB Journal}, 40(4):e71614.

\bibitem[Dunbar et~al., 2026]{dunbar2026modeling}
Dunbar, B., Pramanik, P., and Robinson, H.~K. (2026).
\newblock Modeling age-adjusted mortality in the united states.
\newblock {\em arXiv preprint arXiv:2601.13504}.

\bibitem[Ellington et~al., 2025a]{ellington2025metascorelens}
Ellington, C., Pramanik, P., and Robinson, H.~K. (2025a).
\newblock Metascorelens: Evaluating user feedback across digital entertainment
  systems.
\newblock {\em arXiv preprint arXiv:2601.11523}.

\bibitem[Ellington et~al., 2025b]{ellington2025playmydata}
Ellington, C., Pramanik, P., and Robinson, H.~K. (2025b).
\newblock Playmydata: A statistical analysis of a video game dataset on review
  scores and gaming platforms.
\newblock {\em Analytics}, 4(4):31.

\bibitem[Feynman, 1948]{feynman_1949}
Feynman, R.~P. (1948).
\newblock Space-time approach to non-relativistic quantum mechanics.
\newblock {\em Reviews of Modern Physics}, 20(2):367--387.

\bibitem[{Food and Agriculture Organization of the United Nations},
  2026]{fao_capture_database_2026}
{Food and Agriculture Organization of the United Nations} (2026).
\newblock Global capture production database.
\newblock FAO Fisheries and Aquaculture Statistics.
\newblock Capture fisheries production data by country, species item, and
  fishing area.

\bibitem[Gaudet et~al., 2011]{gaudet2011risk}
Gaudet, M., Press, M., Haile, R., Lynch, C., Glaser, S., Schildkraut, J.,
  Gammon, M., Douglas~Thompson, W., and Bernstein, J. (2011).
\newblock Risk factors by molecular subtypes of breast cancer across a
  population-based study of women 56 years or younger.
\newblock {\em Breast Cancer Research and Treatment}, 130:587--597.

\bibitem[Hening and Tran, 2020]{hening_tran_2020}
Hening, A. and Tran, K.~Q. (2020).
\newblock Harvesting and seeding of stochastic populations: Analysis and
  numerical approximation.
\newblock {\em Journal of Mathematical Biology}, 81(1):65--112.

\bibitem[Hertweck et~al., 2023]{hertweck2023clinicopathological}
Hertweck, K.~L., Vikramdeo, K.~S., Galeas, J.~N., Marbut, S.~M., Pramanik, P.,
  Yunus, F., Singh, S., Singh, A.~P., and Dasgupta, S. (2023).
\newblock Clinicopathological significance of unraveling mitochondrial pathway
  alterations in non-small-cell lung cancer.
\newblock {\em The FASEB Journal}, 37(7):e23018.

\bibitem[Hua et~al., 2019]{hua2019assessing}
Hua, L., Polansky, A., and Pramanik, P. (2019).
\newblock Assessing bivariate tail non-exchangeable dependence.
\newblock {\em Statistics \& Probability Letters}, 155:108556.

\bibitem[Imai et~al., 2010]{imai2010identification}
Imai, K., Keele, L., and Yamamoto, T. (2010).
\newblock Identification, inference and sensitivity analysis for causal
  mediation effects.
\newblock {\em Project euclid}.

\bibitem[Kac, 1949]{kac_1949}
Kac, M. (1949).
\newblock On distributions of certain wiener functionals.
\newblock {\em Transactions of the American Mathematical Society}, 65(1):1--13.

\bibitem[Kakkat et~al., 2023]{kakkat2023cardiovascular}
Kakkat, S., Pramanik, P., Singh, S., Singh, A.~P., Sarkar, C., and Chakroborty,
  D. (2023).
\newblock Cardiovascular complications in patients with prostate cancer:
  Potential molecular connections.
\newblock {\em International Journal of Molecular Sciences}, 24(8):6984.

\bibitem[Kakkat et~al., 2026]{kakkat2026angiotensin}
Kakkat, S., Suman, P., Goswami, S., Kola, B., Bruno, K.~A., Frankel, W.~L.,
  Basu, S., Turbat-Herrera, E.~A., Ramirez-Alcantara, V., Heslin, M.~J., et~al.
  (2026).
\newblock Angiotensin ii type 1 receptor blockade inhibits gastric cancer
  metastasis through tight junction restoration.
\newblock {\em bioRxiv}, pages 2026--01.

\bibitem[Khan et~al., 2023]{khan2023myb}
Khan, M.~A., Acharya, S., Anand, S., Sameeta, F., Pramanik, P., Keel, C.,
  Singh, S., Carter, J.~E., Dasgupta, S., and Singh, A.~P. (2023).
\newblock Myb exhibits racially disparate expression, clinicopathologic
  association, and predictive potential for biochemical recurrence in prostate
  cancer.
\newblock {\em Iscience}, 26(12).

\bibitem[Khan et~al., 2024]{khan2024mp60}
Khan, M.~A., Acharya, S., Kreitz, N., Anand, S., Sameeta, F., Pramanik, P.,
  Keel, C., Singh, S., Carter, J., Dasgupta, S., et~al. (2024).
\newblock Mp60-05 myb exhibits racially disparate expression and
  clinicopathologic association and is a promising predictor of biochemical
  recurrence in prostate cancer.
\newblock {\em Journal of Urology}, 211(5S):e1000.

\bibitem[Lungu and {\O}ksendal, 2001]{lungu_oksendal_2001}
Lungu, E. and {\O}ksendal, B. (2001).
\newblock Optimal harvesting from interacting populations in a stochastic
  environment.
\newblock {\em Bernoulli}, 7(3):527--539.

\bibitem[Maki et~al., 2025]{maki2025new}
Maki, E., Glimm, T., Pramanik, P., Chiari, Y., and Kiskowski, M. (2025).
\newblock New approaches for capturing and estimating variation in complex
  animal color patterns from digital photographs: application to the eastern
  box turtle (terrapene carolina).
\newblock {\em PeerJ}, 13:e19690.

\bibitem[Polansky and Pramanik, 2021]{polansky2021motif}
Polansky, A.~M. and Pramanik, P. (2021).
\newblock A motif building process for simulating random networks.
\newblock {\em Computational Statistics \& Data Analysis}, 162:107263.

\bibitem[Powell and Pramanik, 2025]{powell2025genomic}
Powell, A. and Pramanik, P. (2025).
\newblock Genomic influence of a key transcription factor in male glandular
  malignancy.
\newblock {\em arXiv preprint arXiv:2510.11959}.

\bibitem[Powell and Pramanik, 2026]{powell2026role}
Powell, A. and Pramanik, P. (2026).
\newblock The role of myb in prostate cancer: A statistical analysis.
\newblock {\em European Journal of Statistics}, 6:1--1.

\bibitem[Pramanik, 2016]{pramanik2016}
Pramanik, P. (2016).
\newblock {\em Tail non-exchangeability}.
\newblock Northern Illinois University.

\bibitem[Pramanik, 2020]{pramanik2020optimization}
Pramanik, P. (2020).
\newblock Optimization of market stochastic dynamics.
\newblock {\em SN Operations Research Forum}, 1(4):31.

\bibitem[Pramanik, 2021a]{pramanik2021}
Pramanik, P. (2021a).
\newblock Effects of water currents on fish migration through a feynman-type
  path integral approach under $\sqrt {8/3}$ liouville-like quantum gravity
  surfaces.
\newblock {\em Theory in Biosciences}, 140(2):205--223.

\bibitem[Pramanik, 2021b]{pramanik2021thesis}
Pramanik, P. (2021b).
\newblock {\em Optimization of Dynamic Objective Functions Using Path
  Integrals}.
\newblock PhD thesis, Northern Illinois University.

\bibitem[Pramanik, 2022]{pramanik2022stochastic}
Pramanik, P. (2022).
\newblock Stochastic control of a sir model with non-linear incidence rate
  through euclidean path integral.
\newblock {\em arXiv preprint arXiv:2209.13733}.

\bibitem[Pramanik, 2023a]{pramanik2021consensus}
Pramanik, P. (2023a).
\newblock Consensus as a nash equilibrium of a stochastic differential game.
\newblock {\em European Journal of Statistics}, 3:10--10.

\bibitem[Pramanik, 2023b]{pramanik2023cmbp}
Pramanik, P. (2023b).
\newblock Optimal lock-down intensity: A stochastic pandemic control approach
  of path integral.
\newblock {\em Computational and Mathematical Biophysics}, 11(1):20230110.

\bibitem[Pramanik, 2023c]{pramanik2023cont}
Pramanik, P. (2023c).
\newblock Path integral control in infectious disease modeling.
\newblock {\em arXiv preprint arXiv:2311.02113}.

\bibitem[Pramanik, 2023d]{pramanik2023path}
Pramanik, P. (2023d).
\newblock Path integral control of a stochastic multi-risk sir pandemic model.
\newblock {\em Theory in Biosciences}, pages 1--36.

\bibitem[Pramanik, 2024a]{pramanik2024dependence}
Pramanik, P. (2024a).
\newblock Dependence on tail copula.
\newblock {\em J}, 7(2):127--152.

\bibitem[Pramanik, 2024b]{pramanik2024estimation}
Pramanik, P. (2024b).
\newblock Estimation of optimal lock-down and vaccination rate of a stochastic
  sir model: A mathematical approach.
\newblock {\em European Journal of Statistics}, 4:3--3.

\bibitem[Pramanik, 2024c]{pramanik2024measuring}
Pramanik, P. (2024c).
\newblock Measuring asymmetric tails under copula distributions.
\newblock {\em European Journal of Statistics}, 4:7--7.

\bibitem[Pramanik, 2024d]{pramanik2024estimation1}
Pramanik, P. (2024d).
\newblock On estimation of function-on-function regression kernels with
  brownian berkson errors.

\bibitem[Pramanik, 2024e]{pramanik20242estimation}
Pramanik, P. (2024e).
\newblock On estimation of function-on-function regression kernels with
  brownian berkson errors.

\bibitem[Pramanik, 2024f]{pramanik2024stochastic}
Pramanik, P. (2024f).
\newblock Stochastic control in determining a soccer player’s performance.
\newblock {\em J. Compr. Pure Appl. Math}, 2:111.

\bibitem[Pramanik, 2025a]{pramanik2025construction}
Pramanik, P. (2025a).
\newblock Construction of an optimal strategy: An analytic insight through path
  integral control driven by a mckean--vlasov opinion dynamics.
\newblock {\em Mathematics}, 13(17):2842.

\bibitem[Pramanik, 2025b]{pramanik2025optimal}
Pramanik, P. (2025b).
\newblock Optimal feedback control in social networks in a
  mckean-vlasov-friedkin-johnsen system.
\newblock {\em arXiv preprint arXiv:2508.17138}.

\bibitem[Pramanik, 2025c]{pramanik2025optimal1}
Pramanik, P. (2025c).
\newblock An optimal level of stubbornness to win a soccer match.
\newblock {\em arXiv preprint arXiv:2501.18050}.

\bibitem[Pramanik, 2025d]{pramanik2025stubbornness}
Pramanik, P. (2025d).
\newblock Stubbornness as control in professional soccer games: A bppsde
  approach.
\newblock {\em Mathematics}, 13(3):475.

\bibitem[Pramanik, 2026a]{pramanik2026quantum}
Pramanik, P. (2026a).
\newblock The quantum structure of markets: Linking hamiltonian-jacobi-bellman
  dynamics to schrodinger equation through feynman action.
\newblock {\em arXiv preprint arXiv:2603.25086}.

\bibitem[Pramanik, 2026b]{pramanik2026strategic}
Pramanik, P. (2026b).
\newblock Strategic dynamics of firms via path integral control.
\newblock {\em International Game Theory Review}, page 2650006.

\bibitem[Pramanik et~al., 2024a]{pramanik2024parametric}
Pramanik, P., Boone, E.~L., and Ghanam, R.~A. (2024a).
\newblock Parametric estimation in fractional stochastic differential equation.
\newblock {\em Stats}, 7(3):745.

\bibitem[Pramanik and Dong, 2025a]{pramanik2025impact}
Pramanik, P. and Dong, L. (2025a).
\newblock Impact of random monetary shock: a keynesian case.
\newblock {\em arXiv preprint arXiv:2505.00800}.

\bibitem[Pramanik and Dong, 2025b]{pramanik2025strategic}
Pramanik, P. and Dong, L. (2025b).
\newblock Strategic complementarities due to monetary shock under sticky price.
\newblock {\em European Journal of Statistics}, 5:9--9.

\bibitem[Pramanik et~al., 2024b]{pramanik2024analysis}
Pramanik, P., Graff, J., and DeCaro, M. (2024b).
\newblock Analysis of a tiered pricing model for ecw clients.
\newblock In {\em International Conference on Mathematical Modeling in Physical
  Sciences}, pages 515--535. Springer.

\bibitem[Pramanik et~al., 2025a]{pramanik2025dissecting}
Pramanik, P., Graff, J., and Decaro, M. (2025a).
\newblock Dissecting multi-level pricing schemes in the context of ecw client
  engagement.
\newblock {\em arXiv preprint arXiv:2509.22669}.

\bibitem[Pramanik et~al., 2025b]{pramanik2025factors}
Pramanik, P., Graff, J., and Decaro, M. (2025b).
\newblock On factors influencing consumer preference in pipeline stages: an
  experiment.
\newblock {\em arXiv preprint arXiv:2501.03418}.

\bibitem[Pramanik et~al., 2025c]{pramanik2025strategies}
Pramanik, P., Graff, J., and Decaro, M. (2025c).
\newblock Strategies to increase pipeline status: A case study from eclinical
  data.
\newblock {\em European Journal of Statistics}, 5:3--3.

\bibitem[Pramanik and Maity, 2024]{pramanik2024bayes}
Pramanik, P. and Maity, A.~K. (2024).
\newblock Bayes factor of zero inflated models under jeffereys prior.
\newblock {\em arXiv preprint arXiv:2401.03649}.

\bibitem[Pramanik et~al., 2026]{pramanik2026bayesian}
Pramanik, P., Maity, A.~K., Mandal, A., and Robinson, H.~K. (2026).
\newblock A bayesian discrete framework for enhancing decision-making processes
  in clinical trial designs and evaluations.
\newblock {\em arXiv preprint arXiv:2601.10615}.

\bibitem[Pramanik and Polansky, 2020]{pramanik2020motivation}
Pramanik, P. and Polansky, A.~M. (2020).
\newblock Motivation to run in one-day cricket.
\newblock {\em arXiv preprint arXiv:2001.11099}.

\bibitem[Pramanik and Polansky, 2021]{pramanik2021optimala}
Pramanik, P. and Polansky, A.~M. (2021).
\newblock Optimal estimation of brownian penalized regression coefficients.
\newblock {\em arXiv preprint arXiv:2107.02291}.

\bibitem[Pramanik and Polansky, 2023a]{pramanik2023optimization001}
Pramanik, P. and Polansky, A.~M. (2023a).
\newblock Optimization of a dynamic profit function using euclidean path
  integral.
\newblock {\em SN Business \& Economics}, 4(1):8.

\bibitem[Pramanik and Polansky, 2023b]{pramanik2021scoring}
Pramanik, P. and Polansky, A.~M. (2023b).
\newblock Scoring a goal optimally in a soccer game under liouville-like
  quantum gravity action.
\newblock {\em Operations Research Forum}, 4(3):66.

\bibitem[Pramanik and Polansky, 2023c]{pramanik2023semicooperation}
Pramanik, P. and Polansky, A.~M. (2023c).
\newblock Semicooperation under curved strategy spacetime.
\newblock {\em The Journal of Mathematical Sociology}, pages 1--35.

\bibitem[Pramanik and Polansky, 2024]{pramanik2024motivation}
Pramanik, P. and Polansky, A.~M. (2024).
\newblock Motivation to run in one-day cricket.
\newblock {\em Mathematics}, 12(17):2739.

\bibitem[Valdez and Pramanik, 2025a]{valdez2025association}
Valdez, I. and Pramanik, P. (2025a).
\newblock Association between obesity, race, and luminal subtypes of breast
  cancer.
\newblock {\em European Journal of Statistics}, 5:12--12.

\bibitem[Valdez and Pramanik, 2025b]{valdez2025exploring}
Valdez, I. and Pramanik, P. (2025b).
\newblock Exploring the interplay of adiposity, ethnicity, and hormone receptor
  profiles in breast cancer subtypes.
\newblock {\em arXiv preprint arXiv:2507.21348}.

\bibitem[Vikramdeo et~al., 2024a]{vikramdeo2024abstract}
Vikramdeo, K., Anand, S., Sudan, S., Pramanik, P., Singh, S., Godwin, A.,
  Singh, A., and Dasgupta, S. (2024a).
\newblock Abstract po3-16-05: Mitochondrial dna mutation detection in tumors
  and circulating extracellular vesicles of triple negative breast cancer
  patients for biomarker development.
\newblock {\em Cancer Research}, 84(9\_Supplement):PO3--16.

\bibitem[Vikramdeo et~al., 2024b]{vikramdeo2024mitochondrial}
Vikramdeo, K., Anand, S., Sudan, S., Pramanik, P., Singh, S., Godwin, A.,
  Singh, A., and Dasgupta, S. (2024b).
\newblock Mitochondrial dna mutation detection in tumors and circulating
  extracellular vesicles of triple negative breast cancer patients for
  biomarker development.
\newblock In {\em CANCER RESEARCH}, volume~84. AMER ASSOC CANCER RESEARCH 615
  CHESTNUT ST, 17TH FLOOR, PHILADELPHIA, PA~….

\bibitem[Vikramdeo et~al., 2023]{vikramdeo2023profiling}
Vikramdeo, K.~S., Anand, S., Sudan, S.~K., Pramanik, P., Singh, S., Godwin,
  A.~K., Singh, A.~P., and Dasgupta, S. (2023).
\newblock Profiling mitochondrial dna mutations in tumors and circulating
  extracellular vesicles of triple-negative breast cancer patients for
  potential biomarker development.
\newblock {\em FASEB BioAdvances}, 5(10):412.

\bibitem[{World Bank}, 2026]{worldbank_capture_2026}
{World Bank} (2026).
\newblock Capture fisheries production (metric tons), indicator er.fsh.capt.mt.
\newblock World Development Indicators.
\newblock Source: Food and Agriculture Organization of the United Nations
  (FAO). Data range used in this study: 1960--2023.

\bibitem[Yusuf and Pramanik, 2025a]{yusuf2025predictive}
Yusuf, A. and Pramanik, P. (2025a).
\newblock Predictive significance of cd276/b7-h3 expression in baseline
  biopsies of advanced prostate carcinoma.
\newblock {\em arXiv preprint arXiv:2508.09373}.

\bibitem[Yusuf and Pramanik, 2025b]{yusuf2025prognostic}
Yusuf, A. and Pramanik, P. (2025b).
\newblock Prognostic role of b7-h3 (cd276) expression in initial biopsies of
  metastatic prostate cancer.
\newblock {\em Onco}, 5(3):38.

\end{thebibliography}
\end{document}